\documentclass[11pt]{amsart}
\pdfoutput=1
\usepackage{fullpage}
\usepackage[varg]{txfonts}
\usepackage{amsmath, amssymb,amsaddr}
\usepackage{amsfonts}
\usepackage{mathrsfs}
\usepackage{bbm}
\usepackage{url}
\usepackage[arrow,matrix,curve,cmtip,ps]{xy}
\usepackage{graphicx}
\usepackage{amsthm}
\usepackage{float}
\usepackage{amsthm}
\usepackage[utf8]{inputenc}
\usepackage[T1]{fontenc}
\usepackage[utf8]{inputenc}
\usepackage{multirow}
\usepackage{mathtools}
\usepackage{lipsum}
\usepackage[pagebackref=true,pdfauthor={ShinnYih Huang},citecolor=blue,colorlinks=true,linkcolor=blue]{hyperref}

\allowdisplaybreaks

\newtheorem{theorem}{Theorem}[section]
\newtheorem{lemma}[theorem]{Lemma}
\newtheorem{proposition}[theorem]{Proposition}
\newtheorem{corollary}[theorem]{Corollary}
\newtheorem{conjecture}[theorem]{Conjecture}

\newtheorem*{theorem*}{Theorem}
\theoremstyle{remark}
\newtheorem{remark}[theorem]{Remark}

\numberwithin{equation}{section}

\newcommand{\frb}{\mathfrak{b}}

\newcommand{\rr}{\mathfrak{R}}  
\newcommand{\dd}{\mathfrak{D}}

\newcommand{\Z}{\mathbb{Z}}

\newcommand{\N}{\mathbb{N}}

\newcommand{\R}{\mathbb{R}}
\newcommand{\C}{\mathbb{C}}

\newcommand{\cS}{\mathcal{S}}

\newcommand{\MOmN}{\Omega_N}                
\newcommand{\lmx}{\begin{pmatrix}}                
\newcommand{\rmx}{\end{pmatrix}} 
\newcommand{\dnorm}[1]{ \left\|  #1 \right\| }   
\newcommand{\norm}[1]{ \left |  #1 \right | }     
\newcommand{\tr}{\text{tr}}                             
\newcommand{\lamp}[1]{\lambda_+(#1)}           
\newcommand{\lamm}[1]{\lambda_- (#1)}          
\newcommand{\lam}[1]{\lambda(#1)}                 
\newcommand{\vpl}[1]{v_{+}(#1)}                     
\newcommand{\vmi}[1]{v_{-}(#1)}                     
\newcommand{\powtwo}[1]{{#1}^2}                        
\newcommand{\ppowtwo}[1]{{\left(#1\right)}^2}      
\newcommand{\ppowthree}[1]{{\left(#1\right)}^3}    
\newcommand{\bigO}[1]{O\left( #1 \right)  }            
\newcommand{\frC}{\mathfrak{C}}                          
\newcommand{\frc}{\mathfrak{c}}                            
\newcommand{\frv}{\mathfrak{v}}                            
\newcommand{\calA}{\mathcal{A}}                              
\newcommand{\epcon}{\frr}                             
\newcommand{\pceil}[1]{\left\lceil  #1 \right\rceil  }      
\newcommand{\Ccon}{C_{\epcon}}                   
\newcommand{\calG}{\mathcal{G}}                    
\newcommand{\calC}{\mathcal{C}}                    
\newcommand{\calR}{\mathcal{R}}                   
\newcommand{\frl}{\mathfrak{\ell}}                   
\newcommand{\calY}{\mathcal{Y}}                   
\newcommand{\frq}{\mathfrak{q}}                   
\newcommand{\frM}{\mathfrak{M}}                 
\newcommand{\calI}{\mathcal{I}}                 
\newcommand{\frB}{\mathfrak{B}}                 
\newcommand{\calU}{\mathcal{U}}                 
\newcommand{\calQ}{\mathcal{Q}}                 
\newcommand{\frA}{\mathfrak{A}}                  
\newcommand{\fra}{\mathfrak{a}}                  
\newcommand{\calN}{\mathcal{N}}                  
\newcommand{\calM}{\mathcal{M}}                  
\newcommand{\calE}{\mathcal{E}}                  
\newcommand{\bbj}{\mathbb{j}}                   
\newcommand{\sfC}{\mathsf{C}}                          
\newcommand{\frm}{\mathfrak{m}}                   
\newcommand{\frt}{\mathfrak{t}}                   
\newcommand{\frr}{\mathfrak{r}}                   
\newcommand{\frG}{\mathfrak{G}}                    
\newcommand{\frT}{\mathfrak{T}}                    
\newcommand{\calT}{\mathcal{T}}                    
\newcommand{\inner}[2]{\langle  #1, #2 \rangle}           
\newcommand{\calB}{\mathcal{B}}
\newenvironment{myproof}[1][\proofname]{\proof[#1]\mbox{}\\*}{\endproof}

\title{An Improvement to Zaremba's Conjecture}
\author{ShinnYih Huang}
\email{shinnyih.huang@yale.edu}
\address{Department of Mathematics, Yale University}
\date{\today}
\begin{document}

\begin{abstract}
We prove there exists a  density one subset $\dd \subset \N$ such that each $n \in \dd$ is the denominator of a finite continued fraction with partial quotients bounded by 5. 
\end{abstract}

\maketitle

\setcounter{tocdepth}{1}
\tableofcontents
\section{Introduction}
\vspace{2mm}
\subsection{History of Zaremba's Conjecture.} 
Zaremba's conjecture has been closely related to numerical integration and pseudorandom number generation, see \cite{Nie:78} and \cite{Kon:13}. Several new assertions were made since it was first proposed in 1972. To better understand this conjecture, we introduce following notations.

For $x \in (0,1)$, the integers $a_i(x)$ in the continued fraction expansion of x,
\begin{equation*}\label{eq:contfraction}
x = [a_1,a_2, \ldots, a_k, \ldots] = \frac{1}{ a_1+ \frac{1}{ a_2 + \ddots + \frac{1}{ a_k + \ddots }  } },
\end{equation*}
are called \textbf{partial quotients} of $x$. 

Fix a finite set $\calA \subset \N$, which we call an \textbf{alphabet}. Let $\frC_{\calA}$ denote the collection of irrationals $x \in (0,1)$ with partial quotients $a_i(x)$ of $x$ belonging to the alphabet $\calA$. That is, 
$$
\frC_{\calA} := \left\{  x:  a_i(x) \in \calA \right\}.
$$
The set $\frC_{\calA}$ is a Cantor-like set, and has Hausdorff dimension
$$
\delta_{\calA} = \text{H.dim}(\frC_{\calA}) \in [0,1]. 
$$

Next, let 
$$
\rr_{\calA} := \left\{  \frac{b}{d} : 0 < b< d, (b,d) = 1 \text{, and } \forall i, a_i\left(  \frac{b}{d} \right) \in \calA   \right\}
$$ 
be the set of rationals whose partial quotients belong to $\calA$, and 
$$
\dd_{\calA} := \left\{  d \in \N : \exists (b,d) = 1 \text{ with } \frac{b}{d} \in \rr_{\calA} \right\}
$$
be the set of denominators of fractions in $\rr_{\calA}$.

We now state Zaremba's conjecture.
\begin{conjecture}\label{con:zaremba} 
\emph{(Zaremba \cite{Zar:72})}
There exists some constant $A \in \N$ such that $\dd_{\{1,\ldots,A\}} = \N$.
\end{conjecture}

Bourgain and Kontorovich recently formulated a more general conjecture correcting an earlier conjecture of Hensley \cite{Hensley1996} as follows. For $\calA$ fixed, an integer $d$ is \textbf{admissible} if for all $q > 1$, $\dd_{\calA}$ contains residue $d \pmod{q}$. We denote the set of admissible integers by 
$$
\frA_{\calA} := \left\{  d \in \Z :  \forall q, d \in \dd_{\calA}\pmod{q}  \right\}.
$$ 
Globally, if $d \in \dd_{\calA}$, then we say $d$ is \textbf{represented} by $\calA$. The \textbf{multiplicity} of $d$ is the number of rationals in $\rr_{\calA}$ with $d$ as the denominator.  We then have the following local-global conjecture.

\begin{conjecture}\label{con:bourgain} 
\emph{(\cite[p.~3]{BK:2011})}
If $\delta_{\calA} > 1/2$, then the set of denominators $\dd_{\calA}$ contains every sufficiently large admissible integer.
\end{conjecture}

While Conjecture \ref{con:bourgain} seems to be still out of reach, they prove the following ``almost local-global principle".

\begin{theorem}\label{thm:bourgain} 
\emph{(\cite[p.~3]{BK:2011})}
For dimension $\delta_{\calA} > \delta_0 = 307/312 \approx 0.984$, there exists a subset $\tilde{\dd}_{\calA} \subset \dd_{\calA}$ which contains almost every admissible integer. That is,  for some constant $c = c(\calA) > 0$, we have
\begin{equation}\label{eq:bourgaineffec}
\frac{ \# ( \tilde{\dd}_{\calA} \cap [N/2,N]  )  }{ \# ( \frA_{\calA} \cap [N/2,N] ) } = 1 + O \left( e^{-c\sqrt{\log N}}  \right),
\end{equation}
as $N \rightarrow \infty$. Hence, \eqref{eq:bourgaineffec} holds when $\tilde{\dd}_{\calA}$ is replaced by $\dd_{\calA}$. Furthermore, each $d \in \tilde{\dd}_{\calA}$  appears with multiplicity
\begin{equation}\label{eq:fibercount}
\gg N^{2\delta_{\calA}- \frac{1001}{1000} }.
\end{equation}
The constants $c$ are effectively computable, and the implied constants above depend only on $\calA$.
\end{theorem}

\begin{remark}\label{remark:asympofdelta}
The asymptotic expansion from Hensley \cite{Hen:92},
\begin{equation}\label{eq:asympofdelta}
\delta_{\{ 1,2,\ldots,A  \}}  = 1 - \frac{6}{\pi^2 A} - \frac{72 \log A}{\pi^4 A^2} + O \left( \frac{1}{A^2} \right),
\end{equation}
indicates that Theorem \ref{thm:bourgain} is not vacuous. In fact, Bourgain and Kontorovich showed that $A=50$ is large enough for $\delta_{\calA} > \delta_0 \approx 0.984$. For the sake of presenting the density one statement, they made no effort to optimize the constant $\delta_0$.
\end{remark}

Frolenkov and Kan \cite{FK:2013} later gave a refinement to the constant $\delta_0$ at the cost of a weaker result. In particular, they proved the following positive density statement. 
\begin{theorem}\label{thm:frokan} 
\emph{(\cite{FK:2013})}
For dimension $\delta_{\calA} > \delta _0 = 5/6$, a positive proportion of integers satisfy Zaremba's conjecture. That is,
\begin{equation*}\label{eq:positivedensity}
 \# ( \dd_{\calA} \cap [N/2,N]  )   \gg N.
\end{equation*}
\end{theorem}

In this paper, we combine methods in \cite{BK:2011} and in \cite{FK:2013} to show the following effective density one statement with an improved $\delta_0$ compared to the one in Theorem \ref{thm:bourgain}. 

\begin{theorem}\label{thm:shinnyih}
For dimension $\delta_{\calA} > \delta_0 = 5/6 \approx 0.83333  $, there exists a subset $\tilde{\dd}_{\calA} \subset \dd_{\calA}$ which contains almost every admissible integer. Specifically, there is a constant $c = c(\calA) > 0$ so that
\begin{equation}\label{eq:bourgaineffec2}
\frac{ \# ( \tilde{\dd}_{\calA} \cap [N/2,N]  )  }{ \# ( \frA_{\calA} \cap [N/2,N] ) } = 1 + O \left( e^{-c\sqrt{\log N}}   \right),
\end{equation}
as $N \rightarrow \infty$.  Furthermore, for any sufficiently small fixed constant $\frr$ satisfying $0< \frr < \frac{1}{9}(\delta_{\calA} - \delta_0)$, each $d \in \tilde{\dd}_{\calA}$ produced above appears with multiplicity
\begin{equation}\label{eq:fibercount2}
\gg N^{2\delta_{\calA} -1 - \frr} ,
\end{equation}
as $N \rightarrow \infty$, and the implied constants depend only on $\calA$ and $\frr$. Again, \eqref{eq:bourgaineffec2} remains true when $\tilde{\dd}_{\calA}$ is replaced by $\dd_{\calA}$.
\end{theorem}
\begin{remark}\label{remark:aboutdelta}
Jenkinson \cite{Jen:04} showed that when $\calA = \{ 1,2,3,4,5 \}$, we have $\delta_{\calA} = 0.8368 > 0.8333$. Hence Theorem \ref{thm:shinnyih} is true with this alphabet.
\end{remark}

\begin{remark}\label{remark:verygood}
While our proof works when there are local obstruction, we believe that even with Hausdorff dimension $\delta_{\calA} > 5/6$, the only possible alphabets are those without local obstruction. See Appendix \ref{append:B} for more discussion.
\end{remark}

Moreover, fixing some constant $A$, we say that numbers $x\in \frC_{\calA}$ are \textbf{Diophantine of height $A$} if $\max \calA = A$.  The following Corollary is an improvement to Theorem 1.25 in \cite{BK:2011}. See Appendix \ref{AppendixA} for the proof.

\begin{corollary}\label{Cor:goodapp}
There exist infinitely many primes $p$ with primitive roots $b \pmod{p}$ such that the fractions $b/p$ are Diophantine of height 7.
\end{corollary}

\vspace{2mm}
\subsection{Theorem \ref{thm:shinnyih}: Sketch of Proof} We first  reformulate \eqref{eq:bourgaineffec2}.  Recall a well-kown observation that 
$$
\frac{b}{d} = [a_1,a_2,\ldots,a_k]
$$
is equivalent to 
\begin{equation}\label{equivalentdef}
\lmx \ast & b \\ \ast & d \\ \rmx = \lmx 0 & 1 \\ 1 & a_1 \\ \rmx \lmx 0 & 1 \\ 1 & a_2 \\ \rmx \cdots \lmx 0 & 1 \\ 1 & a_k \\ \rmx.
\end{equation}
Hence,  it is natural for us to focus on the semigroup $\calG_{\calA} \subset \text{GL}(2,\Z)$ generated by matrices 
$
 \lmx 0 & 1 \\ 1 & a \\ \rmx ,
$
for $a \in \calA$. By \eqref{equivalentdef}, we have
\begin{equation}\label{eq:orbit}
\calR_{\calA} = \calG_{\calA} \cdot e_2, \text{ and } \dd_{\calA} =  \langle \calG_{\calA} \cdot e_2,e_2 \rangle.
\end{equation}
Moreover, \eqref{eq:bourgaineffec2} now reads as follows:
$$
\frac{ \# (\langle \calG_{\calA} \cdot e_2,e_2 \rangle \cap [N/2,N]  )  }{ \# ( \frA_{\calA} \cap [N/2,N] ) } = 1 + O \left( e^{-c\sqrt{\log N}}   \right).
$$

The next step is to contruct an exponential sum that allows us to count the appearance of denominators. For that, we define $\calG(N)$ to  be the set of elements in $\calG_{\calA}$ with their Frobenius norm bounded by $N$. That is, 
$$\calG(N) = \left\{ \gamma \in \calG_{\calA}: \dnorm{\gamma} = \sqrt{a^2+b^2+c^2+d^2} < N \right\}.$$
If we write the exponential sum as
\begin{equation}\label{eq:exposumdef101}
S_N(\theta) : = \sum_{\gamma \in \calG(N)} e \left( \langle \gamma e_2, e_2 \rangle         \right),
\end{equation}
then the Fourier coefficients should capture the number of appearance of integers $d$ in $\dd_{\calA}$. 

However, instead of using $\calG(N)$ in \eqref{eq:exposumdef101}, we pick a smaller subset $\Omega_N$ which we now briefly describe its features. For detailed construction, see $\S3.3$.

Let $\Gamma = \Gamma_{\calA}$ be the determinant one subsemigroup of $\calG_{\calA}$. Note that the subsemigroup $\Gamma_{\calA}$ is freely and finitely generated by the matrix products
\begin{equation}\label{eq:generators}
\lmx 0 & 1 \\ 1 & a \\  \rmx \cdot \lmx 0 & 1 \\ 1 & a' \\ \rmx \text{, for } a, a' \in \calA.
\end{equation}
The proof of the following lemma is given in $\S 3.4$.
\begin{lemma}\label{lemma:intro1}
$\Omega_{N} \subset \Gamma_{\calA}$ and 
$$\# \Omega_{N} \gg N^{2\delta - \frac{\frr}{2}}, $$
where the implied constant depends only on $\calA$ and $\frr$.
\end{lemma}

\begin{remark}
We already had by Hensley \cite{Hen:89},
$$
\# \calG(N) \asymp N^{2 \delta}.
$$
Thus, the above Lemma shows that we did not lose too much elements while replacing $\calG(N)$ by $\Omega_N$.
\end{remark}

\begin{remark}
The limit set and Hausdorff dimension of $\calG_{\calA}$ stay the same when we pass to the subsemigroup $\Gamma_{\calA}$. This is due to the fact that the $\calG_{\calA}$-orbit is a finite union of the $\Gamma_{\calA}$-orbits as follows. 
\begin{equation}\label{eq:rrrrrrr}
\calR_{\calA} = \calG_{\calA} \cdot e_2= \Gamma_{\calA} \cdot e_2 \cup_{a \in \calA}  \lmx 0 & 1 \\ 1 & a \\ \rmx \Gamma_{\calA} \cdot  e_2.
\end{equation}
We need this to make sure that the Hausdorff dimension of $\Gamma_{\calA}$ is still near 1. Also, for our convenience, we denote $\lmx 0 & 1 \\ 1 & a \\ \rmx$ by $\gamma_a$, and the set of matrices $ \gamma_a \Gamma_{\calA}$ as $\Gamma_{\calA,a}$. Hence, for some integer $d$, we say that $d$ is \textbf{admissible to $\Gamma_{\calA,a}$} if for all $q >1$, $\langle \Gamma_{\calA,a}e_2, e_2 \rangle$ contains the residue $d \pmod{q}$. 

\end{remark}

We define our new exponential sum as follows. For each integer $a \in \calA$,
\begin{equation}\label{eq:exposumdef10}
S_{N,a}(\theta) : = \sum_{\gamma \in \Omega_{N}} e \left( \langle \gamma e_2, \gamma_a e_2 \rangle         \right) . 
\end{equation}
Note that we also consider 
$$
S_{N}(\theta) : = \sum_{\gamma \in \Omega_{N}} e \left( \langle \gamma e_2, e_2 \rangle         \right).     
$$
We now employ the circle method. Denote the fourier coefficients $R_{N,a}(d)$ of $S_{N,a}(\theta)$ by
\begin{equation}\label{eq:RNdef}
R_{N,a}(d) := \widehat{S}_{N,a}(d) = \int_0^1 S_{N,a}(\theta) e(-d \theta) d \theta =\sum_{\gamma \in \Omega_N} \textbf{1}_{\{ \langle \gamma e_2.\gamma_a e_2   \rangle=d \}},
\end{equation}
and decompose the integral into a ``main term'' and an ``error term'' 
$$R_{N,a}(d) = \calM_{N,a}(d) + \calE_{N,a}(d),$$ 
see $\S 4$ for more detail. Again, we will do all the above procedures for the fourier coefficients $R_{N}(d)$ too.

The main term $\calM_{N,a}(d)$ (or $\calM_{N}(d)$) will be bounded below by 
\begin{theorem}\label{thm:majorterm} 
Let $A = \max{\calA}$. For each $a \in \calA$, and for any integers $d$ admissible to the set $\Gamma_{\calA,a}$ such that $\frac{1}{40}N \leq d < \frac{1}{25}N$, we have
\begin{equation}\label{eq:majorcontribution111}
\calM_{N,a}(d) \gg \frac{1}{\log \log N} \frac{\# \Omega_N}{N},
\end{equation}
where the implied constant depends on $\calA$ and $\frr$.
\end{theorem}
\noindent
On the other hand, the $L^2$-norm of the error term $\calE_{N,a}(d)$ (or  $\calE_{N}(d)$) is bounded by 
\begin{theorem}\label{thm:minorterm}
There exists some $c' = c'(\calA)>0$ such that for each $a \in \calA$, we have
\begin{equation}\label{eq:minorl2bound111}
\sum_{d \in \Z} \norm{\calE_{N,a}(d)}^2 \ll \frac{ {(\#\Omega_N)}^2 }{N} \cdot e^{-c' \sqrt{\log N}},
\end{equation}
where the implied constant depends on $\calA$ and $\frr$.
\end{theorem} 
\noindent
See $\S5$ and $\S9$ for the proofs of above statements.
\vspace{2mm}

Finally, we give the

\noindent
\begin{myproof}[Proof of Theorem \ref{thm:shinnyih}]

For each $a \in \calA$, let $\mathfrak{P}_{a}(N)$ ($\mathfrak{P}(N)$ for the case of $\Gamma_{\calA}$) denotes the set of integers $n \asymp N$ admissible to $\Gamma_{\calA , a}$ which have a small representation number $R_N(n)$. That is,
\begin{equation}\label{eq:smallrep}
\mathfrak{P}_a(N):= \left\{  \frac{1}{40A}N \leq n \leq \frac{1}{25}N : R_{N,a}(n) < \frac{1}{2} \calM_{N,a}(n)  \right\}.
\end{equation}
We now choose $\tilde{\dd}_{\calA}$ to be the complement of $\mathfrak{P}(N) \cup_{a \in \calA} \mathfrak{P}_a(N)$ in the set of admissible integers $n \asymp N$. Moreover, by Theorem \ref{thm:majorterm}, for each $n \in \mathfrak{P}_a(N)$, the error function $\calE_{N,a}(n)$ satisfies
\begin{equation}\label{eq:baderror}
\norm{\calE_{N,a}(n)} = \norm{R_{N,a}(n) - \calM_{N,a}(n)} \gg \frac{1}{\log \log N} \frac{\#\Omega_N}{N},
\end{equation}
and the same goes for $ \mathfrak{P}(N)$. Equipped with Theorem \ref{thm:minorterm}, the size of $\mathfrak{P}_a(N)$, and hence the size of $\mathfrak{P}(N) \cup_{a \in \calA} \mathfrak{P}_a(N)$ is controlled by
\begin{equation*}
\begin{split}
\#\mathfrak{P}_a(N)
& \ll \sum_{ \substack{n  \text{ is admissible to } \Gamma_{\calA,a} \\ \frac{1}{40}N \leq n \leq \frac{1}{25}N  }} \textbf{1}_{ \left\{  \norm{\calE_{N,a}(n)} \gg \frac{\#\Omega_n}{N \log \log N}  \right\}  } \\
& \ll \frac{N^2 \ppowtwo{\log \log N}}{{(\#\Omega_N)}^2} \sum_{n \text{ is admissible to } \Gamma_{\calA,a}} \norm{\calE_{N,a}(n) }^2 \\
& \ll \frac{N^2 \ppowtwo{\log \log N}}{{(\#\Omega_N)}^2} \sum_{n} \norm{\calE_{N,a}(n) }^2 \\
& \ll \frac{N^2 \ppowtwo{\log \log N}}{{(\#\Omega_N)}^2} \cdot \frac{{(\#\Omega_N)}^2}{N} \cdot e^{-c' \sqrt{\log N}} \\
& \ll N \cdot e^{-c\sqrt{\log N}},
\end{split}
\end{equation*}
which together with Theorem \ref{thm:frokan},  implies \eqref{eq:bourgaineffec2}.

Moreover, for each admissible $n \in \left[ \frac{1}{40}N, \frac{1}{25}N  \right] \backslash \left(  \mathfrak{P}(N) \cup_{a \in \calA} \mathfrak{P}_a(N) \right) $, the representation number $R_{N,a}(n)$ (or $R_{N}(n)$) is large. Specifically,
$$
\norm{R_{N,a}(n)} \geq \frac{1}{2} \calM_{N,a}(n) \gg \frac{1}{\log \log N} \frac{\#\Omega_N}{N}  .  
$$
Finally, by Lemma \ref{lemma:intro1}, we conclude that 
$$
\norm{R_{N,a}(n)} \gg N^{2\delta - 1- \frr} , 
$$
as $N \rightarrow \infty $.
\end{myproof}

\begin{remark}\label{remark:1.1}
By strong approximation \cite{MVW:1984}, there exists a bad modulus $\calB$ such that for any $q$ coprime to $\calB$, $\Gamma_{\calA} \cong \text{SL}_2(q) \pmod{q}$. Moreover, for any $q \equiv 0 \pmod{\calB}$, the set $\Gamma_{\calA} \pmod{q}$ is the full preimage of $\Gamma_{\calA} \pmod{\calB}$ under the projection map $\Z / {q} \rightarrow \Z / {\calB}$. Consequently, the set of admissible integers can be fully determined by $\calB$. In addition, when an integer $d$ is not admissible, we no longer have the lower bound for $\calM_{N}(d)$ as stated in Theorem \ref{thm:majorterm}. This is why we need to exclude these integers in the above arguments.
\end{remark}

 In $\S2$, we present several tools for the construction of $\Omega_N$ and arguments in the main term and error term analysis. The set $\Omega_N$ is formulated in $\S3$, assuming the existence of some special set with nice modular and archimedean distribution properties. In $\S4$, we give detailed construction of the main term and error term. The analysis for the main term is carried out in $\S5$. We work on the error term in $\S6$, $\S7$ and $\S8$. Finally, Theorem \ref{thm:minorterm} is proved in $\S9$.
 
\subsection*{Acknowledgement}
We thank Alex Kontorovich for introducing this problem, continuous support, and insightful comments. Moreover, the author acknowledges support from Kontorovich's NSF grants DMS-1209373, DMS-1064214, DMS-1001252, and Kontorovich's NSF CAREER grant DMS-1254788.

\section{Preliminaries}
\subsection{Large Matrix Products} In this section, we review the  large matrix products in \cite{BK:2011}. Recall that $\Gamma = \Gamma_A$ is the semigroup generated by the matrix products
\begin{equation*}
\lmx 0 & 1 \\ 1 & a \\  \rmx \cdot \lmx 0 & 1 \\ 1 & a' \\ \rmx,
\end{equation*}
for $a, a' \in \calA$. By induction, we see that for $\gamma = \lmx a &b \\ c&d\\ \rmx $, $\gamma \neq I$. 
\begin{equation*}
1 \leq a \leq \min(b,c) \leq \max(b,c) <d.
\end{equation*}
That is to say, every non-identity matrix $\gamma \in \Gamma$ is hyperbolic. 
On the other hand, since $d$ is always the largest element in each $\gamma$, the trace, the Frobenius norm, the sup-norm, and the second column norm are all comparable. In particular, we have
\begin{equation}\label{comtracenorm}
\dnorm{\gamma} = \sqrt{a^2+b^2+c^2+d^2} \leq 2 \tr (\gamma) \leq 2 \sqrt{2} \dnorm{\gamma},
\end{equation}
and
\begin{equation}\label{normsupsecond}
\dnorm{\gamma}_{\infty} = d < \norm{\gamma e_2} = \sqrt{b^2+d^2} < \dnorm{\gamma} < \sqrt{2} \norm{\gamma e_2} < 2 \dnorm{\gamma}_{\infty}.
\end{equation}
We use the notations in \cite{BK:2011} for eigenvalues and eigenvectors. For $\gamma \in \Gamma$, let the expanding and contracting eigenvalues of $\gamma$ be $\lamp{\gamma}$ and $\lamm{\gamma} = 1/\lamp{\gamma}$, with corresponding normalized eigenvectors $\vpl{\gamma}$ and $\vmi{\gamma}$. Simple linear algebra and \eqref{comtracenorm} show that
\begin{equation}\label{eq:lambdaandtrace}
\lamp{\gamma} = \tr (\gamma) + O\left(  \frac{1}{ \dnorm{\gamma} } \right),
\end{equation}
where the implied constant is absolute. Write $\lambda = \lambda_+$ for the expanding eigenvalue. 

Note that for all $\gamma \in \Gamma$, the eigenvalues are real, and $\lambda > 1$ if $\gamma \neq I$. We have the following useful results regarding the multiplicity of eigenvalues.
\begin{proposition}\label{Plargenorm} 
\emph{(\cite[p:~10]{BK:2011})}
For every $\gamma \in \Gamma$ sufficiently large, we have
\begin{equation}\label{eq:innerproductbig}
\norm{ \langle \vpl{\gamma} , \vmi{\gamma}^{\perp }    \rangle  } \geq \frac{1}{2}.
\end{equation}
In addition,  the eigenvalues of any two large norm matrices $\gamma, \gamma ' \in \Gamma$ with large norms behave essentially multiplicatively, subject to the directions of their expanding eigenvectors being near to each other. Specifically,
\begin{equation}\label{matrixproduct}
\lam{\gamma \gamma '} = \lam{\gamma} \lam{\gamma '} \left[  1+ \bigO{\norm{\vpl{\gamma} -\vpl{\gamma '}}    +      \frac{1}{\powtwo{\dnorm{\gamma}}}  + \frac{1}{\powtwo{\dnorm{\gamma '}}}      } \right].
\end{equation}
Moreover, the expanding vector of the product $\gamma \gamma '$ faces a nearby direction to that of the first $\gamma$, (and the same in reverse), 
\begin{equation}\label{eigenvectornearby}
\norm{\vpl{\gamma \gamma ' } - \vpl{\gamma}  } \ll \frac{1}{\dnorm{\gamma}^2} \quad \text{and} \quad \norm{\vmi{\gamma \gamma ' } - \vmi{\gamma '}  } \ll \frac{1}{\dnorm{\gamma '}^2}.
\end{equation}
All the implied constants above are absolute.
\end{proposition}

\vspace{2mm}

\subsection{Distributional Properties} In this section, we restate another important result regarding sector counting in the paper \cite{BK:2011}.

Once and for all, we fix the density point $x = [A,A,\ldots,A,\ldots] \in \frC$, where $A = \max \calA$. Notice that when $\calA = \{1,2 \}$, we have $\delta_{\calA}\approx 0.531 < 5/6$ which is known in \cite{PSP:2032896}, \cite{ETS:86119}, and \cite{Bump:1985}. Hence, for the assumption of $\delta_{\calA}$ in Theorem \ref{thm:shinnyih}, we need $A \geq 3$. This implies that 
\begin{equation}\label{eq:xdef}
x = \frac{-A + \sqrt{A^2+4}}{2} < \sqrt{2}-1.
\end{equation}
Let 
\begin{equation}\label{eq:fixedvector}
\frv = \frv_x := \frac{(x,1)}{\sqrt{1+x^2}}
\end{equation}
be the corresponding unit vector of $x$.  One can easily check by \eqref{eq:xdef} that
\begin{equation}\label{eq:vclose}
\langle  \frv , e_2 \rangle = \frac{1}{\sqrt{1+x^2}} > \sqrt{\frac{3}{4}}.
\end{equation}
The following estimate follows from Lalley's methods \cite{Lal:89}. 
\begin{proposition}\label{Psectorcounting} 
\emph{(\cite[p.~13]{BK:2011})}
There is a constant $\frc = \frc(\calA) > 0$ so that as long as $H < e^{\frc \sqrt{\log T}}$, we have
\begin{equation}\label{sectorcounting}
\# \left\{  \gamma \in \Gamma :  \dnorm{\gamma} < T \text{ and } \norm{\vpl{\gamma} - \frv} < \frac{1}{H}     \right\} \gg \frac{T^{2\delta}}{H},
\end{equation}
as $T \rightarrow \infty$. The implied constants depend at most on $\calA$.
\end{proposition}

Bourgain-Gamburd-Sarnak \cite{BGS:2011} later extended the work of Lalley to a congruence setting, and proved the following theorem. 

\begin{theorem}\label{thm:bgs1} 
\emph{(\cite{BGS:2011})} 
There exists an integer
\begin{equation}\label{eq:Bdef}
\frB = \frB(\calA) \geq 1.
\end{equation}
and a constant 
\begin{equation}\label{eq:cdef}
\frc = \frc(\calA) >0
\end{equation}
so that the following holds. Let $\calI$ be the interval centered at $\frv$ with radius $1/H$, and $\mu$ be the $\delta-$dimensional Hausdorff measure on $\frC_{\calA}$ lifted to $\mathbb{P}^1$. For any $(q,\frB) = 1$, any $\omega \in \text{SL}_2(q)$, and any $\gamma_0\in \Gamma$, there is a constant $C(\gamma_0) > 0$ so that  
\begin{equation}\label{eq:coprimeequi}
\begin{split}
& \# \left\{   \gamma \in \Gamma : \gamma \equiv \omega  \pmod{q} \text{, } \norm{\vpl{\gamma} - \frv} < \frac{1}{H}\text{, and } \frac{ \dnorm{\gamma \gamma_0} }{\dnorm{\gamma_0}} \leq T                    \right\}  \\
&= C(\gamma_0) \cdot T^{2\delta} \cdot \frac{\mu (\calI)}{\norm{\text{SL}_2(q)}} + O \left( T^{2\delta } e^{-c\sqrt{\log T}}  \right), \quad \text{as } T \rightarrow \infty.
\end{split}
\end{equation}
With the same setting for $\frB$ and $\frc$, we have, for any $q$ with $\frB \mid q$,
\begin{equation}\label{eq:coprimeequi2}
\begin{split}
& \# \left\{   \gamma \in \Gamma : \gamma \equiv \omega  \pmod{q} \text{, } \norm{\vpl{\gamma} - \frv} < \frac{1}{H}\text{, and } \frac{ \dnorm{\gamma \gamma_0} }{\dnorm{\gamma_0}} \leq T                    \right\}  \\
&= \frac{\norm{\text{SL}_2(\frB)}}{\norm{\text{SL}_2(q)}} \cdot \# \left\{   \gamma \in \Gamma : \gamma \equiv \omega  \pmod{q} \text{, } \norm{\vpl{\gamma} - \frv} < \frac{1}{H}\text{, and } \frac{ \dnorm{\gamma \gamma_0} }{\dnorm{\gamma_0}} \leq T                    \right\} + O \left( T^{2\delta  } e^{-c\sqrt{\log T}}  \right).
\end{split}
\end{equation}
\end{theorem}
We will use Theorem \ref{thm:bgs1} to construct a special set which has good modular distribution. In addition, each element of this special set has its expanding eigenvector close to $\frv$, and its expanding eigenvalue close to some parameter $T$.

\vspace{2mm}

\subsection{Test Functions with Compactly Supported Fourier Transform}\label{testfunction}~

For later purposes, we define a smooth test Function $\Upsilon \in L^1(\R )$ such that its fourier transform $\widehat{\Upsilon }  $ is compactly supported. 

In particular, let $\widehat{F}(x) = \text{rect}(x) = \mathbbm{1}_{[-\frac{1}{2},\frac{1}{2}]} $ be the indicator function of the interval $[-\frac{1}{2},\frac{1}{2}]$. One can easily check that $F(x) = \text{sinc} (x) = \frac{\sin \pi x  }{ \pi x }$. If we take the convolution of $\widehat{F}(x)$ with itself, then we get the triangle function $\psi(x) $ as follows.
\begin{equation}\label{eq:triangle}
\psi (x) := 
\left\{
\begin{array}{l l l}
1+x &  \text{if } -1 < x <0,  \\
1-x &  \text{if } 0 \leq x <1,  \\
0&  \text{otherwise}\\
\end{array} \right.
\end{equation}
Also, we define the function $\Upsilon (x) $ as
\begin{equation}\label{eq:smoothtestfunc}
\Upsilon (x) = \widehat{\psi}(x) =  { \text{sinc} (x)  }^2  .
\end{equation}
This shows that $\widehat{\Upsilon }  (x) = \psi(x)$ is compactly supported and $\Upsilon (x) \in L^1( \R )$. In fact, we can control the support of $\widehat{\Upsilon }  (x)$ by changing $\Upsilon (x)$ to ${ \text{sinc} (a x)  }^2  $.

Similarly, the following seperable function
$$\Upsilon (x,y) = { \text{sinc} (a x)  }^2 { \text{sinc} (a y)  }^2,$$
is a two-dimensional smooth function with compactly supported Fourier transform.

\section{Construction of $\Omega_{N}$}
\subsection{Auxiliary Paramaters $\calN_{j}$} We define an increasing sequence of paramaters $\calN_j$ for the construction of $\Omega_N$, see \eqref{eq:exposumdef10}. Recall the fixed constant $\frr$ given in Theorem \ref{thm:shinnyih} which satisfies the following inequality.
\begin{equation}\label{eq:epcondef}
0< \frr < \frac{1}{9}(\delta_{\calA} - \delta_0).
\end{equation}
Thus, there exists some positive integer $J_1$ such that 
\begin{equation}\label{eq:J1}
 {(1-\epcon)}^{J_1} \leq  \epcon < {(1-\epcon)}^{J_1-1}.
\end{equation}
Note that $J_1$ is of constant size and only depends on $\epcon$. 

We define another parameter $J_2$ as follows.
\begin{equation}\label{eq:J2}
J_2 = \pceil{ \frac{\log \log N  - C_{\epcon ,\calA}}{- \log (1-\epcon )} },
\end{equation}
where the constant $C_{\epcon, \calA}$ depends only on $\calA$ and $\frr$, and is chosen to be large enough so that inequalities \eqref{eq:conCcon}, \eqref{eq:3.3}, \eqref{eigenvaluesonlargeproduct}, \eqref{eigenvaluesonlargeproduct2}, and \eqref{eigenvaluesonlargeproduct3} hold. 
In addition, when $N$ is sufficiently large, we have
\begin{equation}\label{eq:j2j1}
J_2 > 2J_1+2.
\end{equation}
Since $\calA$ is fixed throughout, we denote  $C_{\epcon, \calA}$ by $\Ccon$. Finally, set
\begin{equation}\label{eq:J}
J= J_1+J_2.
\end{equation}

We now construct the auxiliary parameters as
\begin{equation}\label{eq:auxpar}
\calN_{-J}, \ldots, \calN_{-J_1}, \ldots, \calN_{-1}, \calN_0, \calN_1, \ldots, \calN_{J_1}, \ldots, \calN_{J}, \calN_{J+1},
\end{equation}
where $\calN_{J+1} = N$, and 
\begin{equation}\label{DefinitionofN}
\calN_j= \left\{ 
  \begin{array}{l l l l}
    N^{\frac{1}{4}  {(1-\epcon )}^{-j-J_1} } &  \text{,when } -J \leq j \leq -J_1,   \\
    N^{\frac{1}{4} + \frac{1}{4}  {(1-\epcon )}^{-j} }  & \text{,when } -J_1 < j \leq 0,     \\
    N^{\frac{3}{4} -\frac{1}{4} {(1-\epcon )}^{j} }  & \text{,when } 0 \leq j < J_1,     \\    
    N^{1- \frac{1}{4}  {(1-\epcon )}^{j-J_1} }  & \text{,when } J_1 \leq j \leq J.     \\    
    \end{array} \right.
\end{equation}
It is elementary to show that $\calN_0 = \sqrt{N}$, $\calN_{-J_1} = N^{1/4}$, and $\calN_{J_1} = N^{3/4}$.  The following Lemma lists several important properties of the sequence $\{\calN_j\}_{j = -J}^{J+1}$.
\begin{lemma}\label{eq:encemble}
\begin{enumerate}
\item For $-J \leq m \leq J$, we have
\begin{equation}\label{eq:m-m}
\calN_m \calN_{-m} = N.
\end{equation}
\item For $-J\leq m \leq J-1$, we have
\begin{equation}\label{eq:differenceofn}
\frac{\calN_{m+1}}{\calN_m} = 
\left\{
\begin{array}{l l l}
N^{ \frac{\epcon}{4}{(1-\epcon)}^{\norm{m+\frac{1}{2}} -\frac{1}{2} }  }&  \text{,when } -J_1+1\leq m \leq J_1-2,  \\
N^{ \frac{1}{4}{(1-\epcon)}^{J_1-1 }} &  \text{,when } m = -J_1 \text{ or } J_1-1,  \\
N^{ \frac{\epcon}{4}{(1-\epcon)}^{\norm{m+\frac{1}{2}}-J_1 -\frac{1}{2} } }&  \text{,when } m \leq -J_1-1 \text{ or } m\geq J_1,  \\
\end{array} \right.
\end{equation}
and
\begin{equation}\label{eq:sizebigger}
\calN_m \geq \calN_{m+1}^{1-\epcon }.
\end{equation}
\item For $-J\leq m \leq J-1 $, we have
\begin{equation}\label{eq:diffofN}
\frac{\calN_{m+1}}{\calN_m} \geq 2^{ \epcon 2^{ C_{\epcon }-2} },
\end{equation}
and 
\begin{equation}\label{eq:diffofNlast}
2^{ (1-\epcon) 2^{ C_{\epcon }-2 } } \leq \frac{N}{\calN_{J+1}} \leq 2^{  2^{ C_{\epcon }-2 } }.
\end{equation}
\end{enumerate}
\end{lemma}
\begin{proof}
\eqref{eq:m-m}, \eqref{eq:differenceofn}, \eqref{eq:diffofN}, and \eqref{eq:diffofNlast} follow directly from \eqref{eq:J2}, $ J_2 > J_1 $, and the definition of $\calN_j$.
To prove \eqref{eq:sizebigger}, we consider the following cases.
\begin{enumerate}
\item Case $m \leq -J_1-1$ or $-J_1+1 \leq  m \leq -1$. This follows from \eqref{DefinitionofN} directly.
\item Case $m = -J_1, J_1-1$. Straightforward computation shows that we need 
$$
\epcon \geq {(1-\epcon)}^{J_1}\text{, and} \quad 3\epcon \geq {(1-\epcon)}^{J_1-1}
$$
which hold because of  \eqref{eq:J1}.
\item Case $0 \leq m \leq J_1-2$ or $J_1 \leq m$. In this case, we need the following inequality,
$$
\frac{3}{4} - \frac{1}{4} {(1-\epcon)}^m \geq \frac{3}{4} (1-\epcon) - \frac{1}{4} {(1-\epcon)}^{m+2},
$$
which holds since for $\forall m \geq 0 \in \Z$, we have
$$
\frac{3 \epcon}{ {(1-\epcon)}^m} + {(1-\epcon)}^2 \geq 3 \epcon +  {(1-\epcon)}^2 \geq 1.
$$
\end{enumerate}
\end{proof}

The next lemma shows that for any number $M$ sufficiently large, there exists some index $j$ such that $M$ is bounded above and below by $\calN_{j}$ and $\calN_{j-1}$ respectively. Moreover, \eqref{eq:differenceofn} indicates that the upper bound $\calN_{j}$ and lower bound $\calN_{j-1}$ of $M$  are close. 

\begin{lemma}\label{lem:mfallinsize}
For any $M$ with
\begin{equation}\label{eq:sizeofM}
\calN_{-J} \leq M < \calN_{J-1},
\end{equation}
there exist indices $j$ and $h$, such that 
\begin{equation}\label{eq:indicesjh}
-J+1\leq j \leq J-1, \quad h = -j,
\end{equation}
and 
\begin{equation}\label{eq:inequalityofM}
\begin{split}
\calN_{j-1} \leq M \leq \calN_{j}, \quad  \frac{N}{\calN_{h}} \leq M \leq \frac{N}{\calN_{h-1}}.
\end{split}
\end{equation}
Moreover, the above inequalities imply that
\begin{equation}\label{eq:inequalityofM2}
 \calN_{j}^{1-\epcon} \leq M \leq \calN_{j} , \quad 
 {\left( \frac{N}{\calN_{h-1}} \right)}^{1-\epcon} \leq M \leq \frac{N}{\calN_{h-1}}.
\end{equation}
\end{lemma}
\begin{proof}
Since the sequence $\{  \calN_j\}$ is increasing, there exists an index $j$, with $-J+1\leq j \leq J-1$, such that 
$$
\calN_{j-1} \leq M \leq \calN_{j}.
$$
By  \eqref{eq:sizebigger}, we then have
$$
 \calN_{j}^{1-\epcon} \leq  \calN_{j-1} \leq M \leq \calN_{j}.
$$
On the other hand, since $\calN_m \calN_{-m} = N$,  the second inequalities in \eqref{eq:inequalityofM} and \eqref{eq:inequalityofM2} hold.
\end{proof}

The next corollary  is a direct result of  \eqref{eq:differenceofn} and Lemma (\ref{lem:mfallinsize}).
\begin{corollary}\label{cor:mfallinsize}
For any $M$ with
\begin{equation}\label{eq:sizeofM2}
2^{\frac{ 2^{\Ccon-2} }{ 1-\epcon }} \leq M < N^{1-\epcon} < \calN_{J-1},
\end{equation}
there exist indices $j$ and $h$, such that 
\begin{equation}\label{eq:indicesjh2}
-J+1\leq j \leq J-1, \quad h = -j,
\end{equation}
and for which the following inequalities hold.
\begin{equation}\label{eq:inequalityofM3}
\calN_{j}^{1-\epcon} \leq M \leq \calN_{j} , \quad {\left( \frac{N}{\calN_{h-1}} \right)}^{1-\epcon} \leq M \leq \frac{N}{\calN_{h-1}}.
\end{equation}
Moreover,
\begin{equation}\label{inequalityofM4}
M \leq \calN_{j} \leq M^{1+2\epcon} , \quad M \leq \frac{N}{\calN_{h-1}} \leq M^{1+ 2\epcon}.
\end{equation}
\end{corollary}

For later exposition, we set 
\begin{equation}\label{eq:wideC}
\widetilde{C} = \left\lceil 2^{\frac{ 2^{\Ccon-2} }{ 1-\epcon }} \right\rceil .
\end{equation}
In addition, we choose $\Ccon$ large enough so that 
\begin{equation}\label{eq:conCcon}
\widetilde{C} \geq 2^{20}.
\end{equation}

\subsection{The Special Set $\aleph$} 
First, we pick a special parameter $\calN_{\bbj}$ among $\calN_j$'s which is closely related the special set $\aleph$. In particular, taking $M$ to be $N^{2/3}$ in Corollary \ref{cor:mfallinsize}, we obtain an index $\mathbb{j}$ so that
\begin{equation}\label{eq:specialindex}
N^{\frac{2}{3}} \leq \calN_{\mathbb{j}} \leq N^{\frac{2}{3}(1+ 2\epcon)}.
\end{equation}
The index $\mathbb{j}$ is independent of $N$ since its only condition is
\begin{equation}\label{eq:conditiononj}
\frac{3}{4} -\frac{1}{4} {(1-\epcon )}^{\bbj-1} \leq \frac{2}{3} \leq \frac{3}{4} -\frac{1}{4} {(1-\epcon )}^{\bbj}.
\end{equation} 

For ease of presentation, we will assume that for all $q\geq 1$, the reduction of $\Gamma$ is full, 
\begin{equation}\label{eq;fullreduction}
\Gamma \pmod{q} \cong \text{SL}_2(q).
\end{equation}
The general cases are stated as remarks under each theorem. For instance, see Remark \ref{r:111}.

Let 
\begin{equation}\label{eq:bdef}
\frb := \frac{\epcon}{4}{(1-\epcon)}^{\bbj -1},
\end{equation}
and with $R = \norm{\text{SL}_2(\frB)}$, let  $\alpha_0 =  \frac{\frb \frc }{40R}$ be a constant depending only on $\calA$ since the parameters $\frb$ in \eqref{eq:bdef}, $\frB$ in \eqref{eq:Bdef}, and $\frc$ in \eqref{eq:cdef} all depend only on $\calA$. We then set 
\begin{equation}\label{eq:Bdef2}
B := N^{\frb},
\end{equation}
and 
\begin{equation}\label{eq:Qdef}
\calQ := e^{ \alpha_0\sqrt{ \log N}}.
\end{equation}
Recall also the fixed density point $x \in \frC$ with 
\begin{equation*}
\frv = \frv_x := \frac{(x,1)}{\sqrt{1+x^2}}.
\end{equation*}

Let $\calU \subset \left[ \frac{1}{450A}B , \frac{399}{400} B  \right]$
be an arithmetic progression of real numbers starting with $\mu_0 = \frac{1}{450A}B$ having common difference 
\begin{equation}\label{eq:commondiff}
\norm{u - u'} = 2B/\calQ^5,
\end{equation}
for $u,u'$ consecutive terms in $\calU$, and ending with $u > \left( \frac{399}{400} - \frac{2}{\calQ^5} \right) B$. Hence the cardinality of $\calU$ is 
\begin{equation}\label{eq:cardofU}
\norm{\calU} \asymp \calQ^5.
\end{equation}
We now use the Proposition 3.9 in \cite{BK:2011} to construct the special set $\aleph$. For detailed proof, we need Theorem \ref{thm:bgs1} and the random extraction argument in \cite[$\S8.2$]{BK:2011}. Notice that the constants in this Proposition are not exactly the same but the proof only requires minor changes.
\begin{proposition}\label{prop:aleph} 
\emph{(\cite[p.~12]{BK:2011})} 
For each $u \in \calU$, there are non-empty sets $\aleph_u \subset \Gamma $, all of the same cardinality
\begin{equation}\label{eq:samecard}
\norm{ \aleph_{u} } = \norm{  \aleph_{u'} },
\end{equation}
so that the following holds. For every $\fra \in \aleph_{u}$, its expanding eigenvector is restricted by 
\begin{equation}\label{eq:alepheig}
\norm{ \vpl{ \fra}  - \frv} < \calQ^{-5},
\end{equation}
and its expanding eigenvalue $\lambda(\fra)$ is restricted by 
\begin{equation}\label{eq:alepheigv}
\norm{ \lambda(\fra) - u } < \frac{B}{\calQ^{5}}.
\end{equation}
In particular, 
\begin{equation}\label{eq:rangeofeig}
\frac{1}{500A} B < \lambda(\fra) < B,
\end{equation}
for $N$ sufficiently large. Moreover, for any $q < \calQ$, any $\omega \in \text{SL}_2(q)$, and any $u \in \calU$, we have
\begin{equation}\label{eq:equivaleph}
\# \left\{    \fra \in \aleph_u : \fra \equiv \omega \pmod{q}              \right\}
= \frac{ \norm{\aleph_u} }{ \norm{\text{SL}_2(q)} } (1+O(\calQ^{-4})),
\end{equation}
where the implied constant does not depends on $q$, $\omega$, or $u$.
\end{proposition}
\begin{remark}\label{r:111}
When there is local obstruction, for any $q$, let $\mathcal{S}_q$ be the set of admissible matrices $\omega \in \text{SL}_2(q)$. That is to say, 
$$
\mathcal{S}_q = \left\{  \omega \in \text{SL}_2(q) : \exists \gamma \in \Gamma \text{ s.t. }  \gamma \equiv \omega   \pmod{q}      \right\}.
$$
Then instead of \eqref{eq:equivaleph}, for any $q < \calQ$, any $\omega \in \mathcal{S}_q$, and any $u \in \calU$, we have
$$
\# \left\{    \fra \in \aleph_u : \fra \equiv \omega \pmod{q}              \right\}
= \frac{ \norm{\aleph_u} }{ \norm{\mathcal{S}_q }} (1+O(\calQ^{-4})),
$$

\end{remark}

With the sets $\aleph_u$ formulated as above, we define the special set $\aleph$ to be the union of them,
\begin{equation}\label{eq:realaleph}
\aleph := \bigsqcup_{u \in \calU}   \aleph_u.
\end{equation}
Note that $\aleph_u$ are disjoint because of  \eqref{eq:commondiff} and  \eqref{eq:alepheigv}.
\vspace{2mm}

\subsection{Construction of  $\MOmN$}\label{decompose} We first need the following Proposition of which the proof uses Proposition \ref{Psectorcounting} and pigeonhole argument. 
\begin{proposition}\label{prop:Xi} 
\emph{(\cite[p.~14]{BK:2011})} 
Given $M \gg 1$ and $H < e^{\frc \sqrt{\log M}}$ (the constant $\frc$ is given in Proposition \ref{Psectorcounting}), there exists some $L$ in the range 
\begin{equation}\label{eq:rangeofL}
\frac{1}{4}M \leq L \leq 4M,
\end{equation}
an integer $k \asymp \log M$, and a set $\Xi = \Xi(M,H;L,k) \subset \Gamma $ having the following properties. 
\begin{enumerate}
\item For $\gamma \in \Xi$, the expanding eigenvalues are controlled to within $1/\log L$:
\begin{equation}\label{eq:eigenclose}
L\left( 1-\frac{1}{\log L}  \right) < \lambda(\gamma) < L.
\end{equation}
\item For $\gamma \in \Xi$, the expanding eigenvectors are controlled to within $1/H$:
\begin{equation}\label{eq:eigenvectorclose}
\norm{ \vpl{\gamma} - \frv  } < \frac{1}{H}.
\end{equation}
\item For $\gamma \in \Xi$, the wordlength metric $\ell $ (in the generators \eqref{eq:generators} of $\Gamma$) satisfies
$$
\ell(\gamma) = k.
$$
\item The cardinality of $\Xi$ is bounded by 
\begin{equation}\label{eq:sizeofxi}
\frac{L^{2\delta}}{H \ppowtwo{\log L}  } \ll \# \Xi \ll L^{2\delta}.
\end{equation}
\end{enumerate}
\end{proposition}
\vspace{2mm}

\begin{remark}
We will use the above Proposition to build up $\MOmN$. In particular, we usually take $H = \log M$ which for $M$ sufficiently large,
\begin{equation}\label{eq:3.3}
\log M < e^{c \sqrt{\log M}}.
\end{equation}
Moreover, $M$ will be bounded below by $2^{ \epcon 2^{ C_{\epcon }-2}  - 2}$, see \eqref{eq:diffofN}. Hence, we need $C_{\epcon }$ to be large enough so that \eqref{eq:3.3} holds.
\end{remark}

We construct the set $\MOmN$ as follows.
\vspace{2mm}

\textbf{Setup: } First take
\begin{equation}\label{biginnerM0}
M_{-J} = \calN_{-J} \geq 2^{2^{\Ccon -1}  }, \quad H = \log M_{-J},
\end{equation}
and use Proposition \ref{prop:Xi} to generate a set $\Xi(M_{-J},H;L_{-J},k)$. We also write $\Xi(M_{-J},H;L,k)$ as $\Xi_{-J}$. Notice that we have 
\begin{equation}\label{L0}
L_{-J} = a_{-J} \calN_{-J},
\end{equation}
for some $\alpha_{-J} \in (1/4,4)$, and
$$
\#\Xi_{-J} \gg \frac{L_{-J}^{2\delta}}{  \ppowthree{\log L_{-J}}}.
$$
\vspace{2mm}

\textbf{Step 1: } Next we set 
$$
M_{-J+1} = \frac{\calN_{-J+1}}{L_{-J}} = \frac{\calN_{-J+1}}{a_{-J} \calN_{-J}} > 2^{ \epcon 2^{ C_{\epcon }-2}  - 2}, \quad H = \log M_{-J+1}, 
$$
and generate another set $\Xi(M_{-J+1},H;L_{-J+1},k)$, denoted by $\Xi_{-J+1}$. Again by Proposition \ref{DefinitionofN}, we have $L_{-J+1} =a_{-J+1}M_{-J+1} $, for some $a_{-J+1} \in (1/4,4)$, and 
$$
\#\Xi_{-J+1}  \gg \frac{L_{-J+1}^{2\delta}}{  \ppowthree{\log L_{-J+1}}}.
$$
\vspace{2mm}

\textbf{Iterate: } Start with $j=2-J$ and iterate up to $j = \bbj-1$, as defined in \eqref{eq:specialindex}. For each such $j$, set 
\begin{equation}\label{eq:iteratingM}
M_j := \frac{\calN_{j}}{ a_{j-1}\calN_{j-1} } > 2^{ \epcon 2^{ C_{\epcon }-2}  - 2}, \quad H = \log M_j,
\end{equation}
and generate a set $\Xi (M_j,H;L_j,k)$, denoted by $\Xi_j$. Note that $L_j = a_j M_j$, with $a_j \in (1/4,4)$, and 
\begin{equation}\label{IteratingXi}
\#\Xi_j  \gg \frac{L_j^{2\delta}}{  \ppowthree{\log L_j}}.
\end{equation}
\vspace{2mm}

\textbf{Special Set $\aleph$: } Recall that we have presupposed the existence of a set $\aleph$ in \eqref{eq:realaleph}, all of whose expanding eigenvectors are within $\calQ^{-5}$ ($\calQ$ is defined in \eqref{eq:Qdef}) of $\frv$, and with eigenvalues of size $B$,  see \eqref{eq:rangeofeig}. From \eqref{eq:bdef} and \eqref{eq:Bdef}, we have 
$B =N^{\frb} =  \calN_{\bbj} / \calN_{\bbj-1}$. For the sake of convenience, the symbol $L_{\bbj}$ also represents $B$ in the later context.
\vspace{2mm}

\textbf{After the Special Set: } By the definition of $\alpha_0$ and \eqref{eq:Qdef},  we have 
$$
\calQ^5   = e^{\frac{\frac{\epcon \frc}{4}{(1-\epcon)}^{\bbj -1} \sqrt{\log N}}{8R}}< e^{\frc \sqrt{ \log M_{\bbj +1}}},
$$
where
\begin{equation}\label{eq:afterspecial}
M_{\bbj+1} := \frac{\calN_{\bbj+1}}{a_{\bbj-1}\calN_{\bbj}} = \frac{N^{\epcon {(1-\epcon)}^{\bbj}}}{4 a_{\bbj-1} }.
\end{equation}
Using Proposition \ref{prop:Xi}  with $M = M_{\bbj+1}$ and $H = \calQ^5$, we obtain a set $\Xi (M_{\bbj+1},H;L_{\bbj+1},k)$, denoted by $\Xi_{\bbj+1}$. Note that $L_{\bbj+1}= a_{\bbj+1} M_{\bbj+1}$, for some $a_{\bbj+1} \in (1/4,4)$. In addition, the cardinality of $\Xi_{\bbj+1}$ is bounded by
\begin{equation}\label{IteratingXisecond}
\#\Xi_{\bbj+1}  \gg \frac{L_{\bbj+1}^{2\delta}}{  {\calQ}^5\ppowtwo{\log L_{\bbj+1}}}.
\end{equation}
\vspace{2mm}

\textbf{Iterate Again: } Start with $j=\bbj+2$ and iterate up to $j = J$, as defined in \eqref{eq:specialindex}. For each such $j$, set 
\begin{equation}\label{eq:iteratingM2}
M_j := \frac{\calN_{j}}{ a_{j-1}\calN_{j-1} }, \quad H = \log M_j,
\end{equation}
and generate a set $\Xi (M_j,H;L_j,k)$, denoted by $\Xi_j$. Note that $L_j = a_j M_j$, with $a_j \in (1/4,4)$, and 
\begin{equation}\label{eq:IteratingXi2}
\#\Xi_j  \gg \frac{L_j^{2\delta}}{  \ppowthree{\log L_j}}.
\end{equation}
\vspace{2mm}

\textbf{End: } For the last step, $j=J+1$, we set
$$
M_{J+1} :== \frac{N}{a_{J} N_{J} }, \quad H = \log M_{J+1},
$$
and generate the last set $\Xi_{J+1} = \Xi(M_{J+1} ,H;L_{J+1},k )$. We have the last parameter $L_{J+1} = a_{J+1} M_{J+1}$, with $a_{J+1} \in (1/4,4)$, and
\begin{equation}\label{eq:lastXi}
\#\Xi_{J+1}  \gg \frac{L_{J+1}^{2\delta}}{  \ppowthree{\log L_{J+1}}}.
\end{equation}
We now define $\MOmN$ by concatenating the sets $\Xi_j$ developed above. That is, 
\begin{equation}\label{eq:omegaNdefinition}
\MOmN := \Xi_{-J} \cdot \Xi_{-J+1} \cdots \Xi_{\bbj-1} \cdot \aleph \cdot \Xi_{\bbj+1} \cdot \Xi_{\bbj+2} \cdots  \Xi_{J} \cdot \Xi_{J+1}.
\end{equation}

\vspace{2mm}
\subsection{Properties of $\Omega_N$}

For $\gamma \in \MOmN$, write
$$
\gamma = \xi_{-J} \xi_{-J+1} \cdots \fra \cdots  \xi_{J+1}
$$
according to the decomposition (\ref{eq:omegaNdefinition}), where $\fra \in \aleph$, and $\xi_j \in \Xi_j$ for $\forall j$. In addition, by the fixed wordlength restriction, the decomposition is unique. (Start from both tails, and gradually determine all the $\xi_j$.) First of all, we have 

\begin{lemma}\label{LsumboundforlogL}
For any arbitrarily large constant $\sfC$, we can choose the constants $\Ccon$ and $N$ sufficiently large so that
\begin{equation}\label{sumsizeaboutLj}
\frac{2}{\calQ^5}+\sum_{\substack{ j=-J \\ j \neq \bbj, \bbj+1 } }^{J+1} \frac{1}{\log L_j} < \frac{1}{\sf{C}}.
\end{equation}
\end{lemma}
\begin{remark}
We shall specify the bound of $\sfC$ later, see \eqref{eq:lambdainduction1} and \eqref{eq:lambdainduction2}. 

\end{remark}

\begin{proof}
From the construction of $\MOmN$, we have
\begin{equation}\label{sizeofL0}
L_{-J} = a_{-J} M_{-J},
\end{equation}
and for $j \neq \bbj, \bbj+1$, we have 
\begin{equation}\label{sizeofL}
L_j = \frac{a_j \calN_{j}}{a_{j-1} \calN_{j-1}} = \frac{a_j}{\alpha_{j-1}} 
\left\{
\begin{array}{l l l}
N^{ \frac{\epcon}{4}{(1-\epcon)}^{\norm{j-\frac{1}{2}} - J_1 -\frac{1}{2} } }&  \text{when } -J+1 \leq j \leq -J_1 \text{ or }  J_1+1\leq j \leq J,  \\
N^{ \frac{1}{4}{(1-\epcon)}^{J_1-1 }}&  \text{when } j = -J+1_1 \text{ or } j =  J_1,  \\
N^{ \frac{\epcon}{4}{(1-\epcon)}^{\norm{j-\frac{1}{2}} -\frac{1}{2} } }&  \text{when } 2-J_1 \leq j  \leq J_1-1. \\
\end{array} \right.
\end{equation}
By the fact that 
$
\frac{\calN_{m+1}}{\calN_m} \geq 2^{ \epcon 2^{ C_{\epcon }-2 } }
$
and $J_2 > J_1$,  for $\Ccon$ sufficiently large, the following inequalities hold.
\begin{equation}\label{boundforallLogL}
\begin{split}
\sum_{ \substack{ j=-J \\ j \neq \bbj, \bbj+1 }  }^{J+1} \frac{1}{  \log   L_j}  
& \leq 4\cdot \frac{4}{\epcon} \cdot \frac{2}{\log N} \cdot \left(  \frac{1}{{(1-\epcon)}^{J_2}}    +  \frac{1}{{(1-\epcon)}^{J_2-1}} + \cdots +  \frac{1}{{(1-\epcon)}} + 1  \right) \\
& \leq \frac{32}{\log N} \cdot \frac{1}{ \epcon{(1-\epcon)}^{J_2} } \cdot \left(  1+ (1-\epcon) + 
{(1-\epcon)}^2 + \cdots   \right) \\
& \leq \frac{1}{\log N} \cdot \frac{32}{ \powtwo{\epcon}{(1-\epcon)}^{J_2}} \\
& \leq \frac{32}{\epcon^2 (1-\epcon) 2^{C_{\epcon}}}.
\end{split}
\end{equation}
Therefore, the constant $\sfC$ can be arbitrarily large depending on $\Ccon$ and $N$.
\end{proof}

The next Lemma gives an upper and lower bound to products of $L_j$'s.
\begin{lemma}\label{lemma:productofL}
For any $-J \leq j \leq h \leq J+1$, we have
\begin{equation}\label{eq:productofL1}
\frac{1}{4}< \frac{L_{-J} L_{-J+1} \cdots L_{h} }{\calN_{h}} <4
\end{equation}
and 
\begin{equation}\label{eq:productofL2}
\frac{1}{16}< \frac{L_{j} L_{j+1} \cdots L_{h} }{\calN_{h}/ \calN_{j-1}} <16
\end{equation}
\end{lemma}
\begin{proof}
This follows directly from the definition of $L_j$ in $\S 3.3$.
\end{proof}

We now use Lemma \ref{LsumboundforlogL} to show that we have control on the eigenvalues and eigenvectors of products of $\Xi_j$'s. 
\begin{lemma}\label{lemma:controloneigenvalues}
For any $-J\leq j \leq h\leq J+1$, and $\fra \in \aleph$, $\xi_j\in \Xi_j, \ldots, \xi_h\in \Xi_h$, we have 
\begin{equation}\label{eigenvaluesonlargeproduct}
\frac{1}{2} <  \frac{ \lam{\xi_j \xi_{j+1} \cdots \xi_{h-1} \xi_h }    }{ L_j L_{j+1} \cdots L_{h-1} L_h  }  < 2,
\end{equation}
when $j>\bbj+1 \text{ or } h < \bbj$.

Similarly, 
\begin{equation}\label{eigenvaluesonlargeproduct2}
\frac{1}{1000} <  \frac{ \lam{\xi_j \xi_{j+1} \cdots \fra \cdots \xi_{h-1} \xi_h }    }{ L_j L_{j+1} \cdots B \cdots L_{h-1} L_h  }  < 2,
\end{equation}
and
\begin{equation}\label{eigenvaluesonlargeproduct3}
\frac{1}{2} <  \frac{ \lam{\xi_j \xi_{j+1} \cdots \xi_{\bbj-1} \xi_{\bbj+1} \cdots \xi_{h-1} \xi_h }    }{ L_j L_{j+1} \cdots L_{\bbj-1} L_{\bbj+1} \cdots L_{h-1} L_h  }  < 2.
\end{equation}
In addition, the eigenvectors of products of $\xi_j$ are close to $\frv$.  Specifically,  for $j \neq \bbj+1$,
\begin{equation}\label{eq:eigenvectoronlargeproduct1}
\norm{ \vpl{\xi_j\xi_{j+1} \cdots \xi_{h-1}\xi_h } - \frv } \ll \frac{1}{\log L_j },
\end{equation}
and 
\begin{equation}\label{eq:eigenvectoronlargeproduct2}
\norm{ \vpl{\xi_{\bbj+1} \cdots \xi_{h-1}\xi_h } - \frv } \ll \frac{1}{\calQ^5},
\end{equation}
where the implied constant depends only on $\calA$. 

In fact, since 
$$
L_j L_{j+1} \cdots L_{h-1}L_h =  \left\{ 
  \begin{array}{l l}
    \alpha_{h}\calN_{h} & \quad \text{if } j = -J,\\
   \frac{ \alpha_{h}\calN_{h} }{  \alpha_{j-1} \calN_{j-1}  } & \quad \text{if } j > -J,
  \end{array} \right.
  $$
  we have
  \begin{equation}\label{eq:eigenvaluesonlargeproduct3}
  \frac{1}{8} <  \frac{ \lam{\xi_{-J} \xi_{-J+1} \cdots \xi_{\bbj-1} \xi_{\bbj+1} \cdots \xi_{J-1} \xi_J }    }{ N/B  }  < 8,
  \end{equation}
and for any $j > -J$ and any $h \geq j$, we have
\begin{equation}\label{eigenvaluesonlargeproduct4}
  \frac{1}{16000} <  \frac{ \lam{\xi_j \xi_{j+1} \cdots \xi_{h-1} \xi_h }    }{ \calN_{h}/ \calN_{j-1}  }  < 32.
\end{equation}
\end{lemma}
\begin{proof}
We mimick the proof of Lemma 3.38 in \cite{BK:2011}. First of all, \eqref{eq:eigenvectoronlargeproduct1} and \eqref{eq:eigenvectoronlargeproduct2} follows directly from \eqref{eigenvectornearby}, \eqref{eq:afterspecial}, and the construction of $\Omega_N$ in $\S 3.4$. Take $\norm{ \vpl{\xi_j\xi_{j+1} \cdots \xi_{h-1}\xi_h } - \frv }$ as an example, we have by Proposition \ref{Plargenorm} that
\begin{equation}\label{eq:vcloseproof}
\begin{split}
\norm{ \vpl{\xi_j\xi_{j+1} \cdots \xi_{h-1}\xi_h } - \frv } 
& \leq \norm{ \vpl{\xi_j\xi_{j+1} \cdots \xi_{h-1}\xi_h } - \vpl{\xi_j} } + \norm{ \vpl{\xi_j} -\frv } \ll \frac{1}{\log L_j},
\end{split}
\end{equation}
where the implied constant is absolute. For \eqref{eigenvaluesonlargeproduct}, we are able to prove by  \eqref{eq:eigenvectoronlargeproduct1} and downward induction on $j$ that
\begin{equation}\label{eq:lambdainduction1}
\begin{split}
\lam{\xi_j \xi_{j+1} \cdots \xi_{h-1} \xi_h }  
&= L_{j} L_{j+1} \cdots L_{h-1} L_{h} \\
& \quad \times \left[1+ O\left(  \frac{1}{\log L_j} + \frac{1}{\log L_{j+1}} + \cdots + \frac{1}{\log L_{h}}  \right)          \right] \\
& = L_{j} L_{j+1} \cdots L_{h-1} L_{h} \left[  1+ O\left(   \frac{1}{\sfC}  \right)   \right],
\end{split}
\end{equation}
where the implied constant only depends on $\calA$, and the constant $\sfC$ is from Lemma \eqref{LsumboundforlogL}. Hence we need the constant $\sfC$ to be large enough to beat the implied constant so that \eqref{eigenvaluesonlargeproduct} is true.

Similarly, for \eqref{eigenvaluesonlargeproduct2}, we have the following equation
\begin{equation}\label{eq:lambdainduction2}
\begin{split}
\lam{\xi_j \xi_{j+1} \cdots \fra \cdots \xi_{h-1} \xi_h }  
&= L_{j} L_{j+1} \cdots \lambda(\fra) \cdots L_{h-1} L_{h} \\
& \quad \times \left[1+ O\left( \frac{2}{\calQ^5} + \frac{1}{\log L_j} + \frac{1}{\log L_{j+1}} + \cdots + \frac{1}{\log L_{h}}  \right)          \right] \\
& =L_{j} L_{j+1} \cdots \lambda(\fra) \cdots L_{h-1} L_{h} \left[  1+ O\left(   \frac{1}{\sfC}  \right)   \right].
\end{split}
\end{equation}
Thus, again we want the constant $\sfC$ in Lemma \eqref{LsumboundforlogL} sufficiently large so that \eqref{eigenvaluesonlargeproduct2} follows from \eqref{eq:rangeofeig}. We can prove the last equation \eqref{eigenvaluesonlargeproduct3} using similar arguments as above. 

Finally, combining \eqref{eigenvaluesonlargeproduct}, \eqref{eigenvaluesonlargeproduct2}, and \eqref{eigenvaluesonlargeproduct3} with \eqref{eq:productofL1} and \eqref{eq:productofL2}, we prove \eqref{eq:eigenvaluesonlargeproduct3} and \eqref{eigenvaluesonlargeproduct4}.
\end{proof}

Next, we need the following observation to control the size of products of $\Xi_j$'s.

\begin{lemma}\label{sizecontrolproduct}
For any $-J\leq j \leq J+1$, and $N$ sufficiently large, we have
\begin{equation}\label{boundofproductlogL}
\log L_{-J} \log L_{-J+1} \cdots \log L_j \leq 2^{  \frac{12 \ppowtwo{\log \log \calN_{j}} }{- \log (1- \epcon)}    }.
\end{equation}
Similarly, for any $-J\leq h \leq J$, we have
\begin{equation}\label{boundofproductlogL2}
\log L_{h+1}\log L_{h+2} \cdots \log L_{J+1} \leq 2^{  \frac{12 \ppowtwo{\log \log (N/\calN_{h})} }{- \log (1- \epcon)}    }.
\end{equation}
\end{lemma}
\begin{proof}
Here we give the proof of \eqref{boundofproductlogL}.  Notice that since the magnitude of $L_j$ and $L_{-j+1}$ are the same (off by bounded constants $\alpha_j$'s only),  the proof of \eqref{boundofproductlogL2} is the same as the one of \eqref{boundofproductlogL}. Let us consider the following two cases.

\textbf{Case} $j \leq -J_1$. By the formula of $L_j$ in \eqref{sizeofL0} and \eqref{sizeofL}, we have
$$
\log \log L_j \leq \log \log N + (-i-J_1) \log (1-\epcon).
$$
Summing over $j$, we obtain
\begin{equation}\label{eq:sumloglog}
\log \log L_{-J} + \log \log L_{-J+1} + \cdots + \log \log L_j \leq (j+J+1) \left[ \log \log N + (-j-J_1) \log (1-\epcon)  \right].
\end{equation}
We bound the right hand side of \eqref{eq:loglogN} as follows. 

For $N$ large enough, \eqref{DefinitionofN} implies that
\begin{equation}\label{eq:loglogN}
\left[ \log \log N + (-j-J_1) \log (1-\epcon)  \right] \leq 2 \log \log \calN_j .
\end{equation}
Also, the term $j+J+1$ is bounded as follows.
\begin{equation}\label{eq:loglogN2}
\begin{split}
(j+J+1)\leq J_2+1+j+J_1 \leq \left(  \frac{ \log \log N }{ - \log (1-\epcon) }   + j+J_1 \right) \leq \frac{2\log \log \calN_j } {- \log (1-\epcon)},
\end{split}
\end{equation}
where the second inequality comes from \eqref{eq:J2}. 

\eqref{eq:sumloglog}, \eqref{eq:loglogN}, and \eqref{eq:loglogN2} now imply \eqref{boundofproductlogL}.

\textbf{Case} $j > -J_1$. Then we have $\calN_j \geq N^{1/4}$. Thus, for $N$ large enough, 
\begin{equation}\label{eq:loglogN3}
\log \log \calN_j \geq \frac{1}{2} \log \log N.
\end{equation}
Similar to \eqref{eq:sumloglog}, we have
\begin{equation}\label{eq:sumloglogL2}
\log \log L_{-J} + \log \log L_{-J+1} + \cdots + \log \log L_j \leq  (j+J+1) \log \log N.
\end{equation}
Again \eqref{eq:J2} shows
\begin{equation}\label{eq:jJ}
j+J+1 \leq 2J+2 \leq 3J_2 \leq \frac{ 3\log \log N}{ - \log (1-\epcon)}.
\end{equation}
Combining \eqref{eq:loglogN3}, \eqref{eq:sumloglogL2}, and \eqref{eq:jJ}, we then prove \eqref{boundofproductlogL}.
\end{proof}
Finally, we close this section with the

\noindent
\begin{myproof}[Proof of Lemma \ref{lemma:intro1}]

Recall from \eqref{eq:omegaNdefinition} that 
$$
\MOmN := \Xi_{-J} \cdot \Xi_{-J+1} \cdots \Xi_{\bbj-1} \cdot \aleph \cdot \Xi_{\bbj+1} \cdot \Xi_{\bbj+2} \cdots  \Xi_{J} \cdot \Xi_{J+1}.
$$
From the construction of $\Xi_j$ in $\S$ \ref{decompose}, we have by crudely using $\norm{\aleph} \geq 1$ and $\norm{\Xi_{\bbj+1}} \geq 1$,
\begin{equation}
\begin{split}
\#\MOmN \geq   \frac{N^{2\delta}}{ L_{\bbj}^{2\delta} L_{\bbj+1}^{2\delta}} \cdot \frac{1}{ { \left( \log L_{-J} \log L_{-J+1} \cdots \log L_{J+1}  \right)     }^4}.
\end{split}
\end{equation}
Applying Lemma \ref{sizecontrolproduct}, and then we get
\begin{equation}\label{eq:888}
\begin{split}
\#\MOmN 
& \geq   \frac{N^{2\delta}}{ L_{\bbj}^{2\delta} L_{\bbj+1}^{2\delta}} \cdot  2^{  \frac{48 \ppowtwo{\log \log \calN_{j}} }{ \log (1- \epcon)}    }  \geq   \frac{N^{2\delta}}{ N^{4\delta \epcon / 11}  } \cdot  2^{  \frac{48 \ppowtwo{\log \log \calN_{j}} }{ \log (1- \epcon)}    }, \\
\end{split}
\end{equation}
where the second inequality is true because $\frac{1}{4} {(1-\epcon )}^{\bbj-1} < \frac{1}{11}$, see \eqref{eq:conditiononj}.
Also, notice that 
$$
N^{-\frac{4\delta \epcon}{11}} \cdot 2^{  \frac{48 \ppowtwo{\log \log \calN_{j}} }{ \log (1- \epcon)}    } \gg N^{- \frac{\epcon}{2}},
$$
where the implied constant depends on $\epcon$. Thus, together with  \eqref{eq:888}, we have
$$
\#\MOmN  \gg N^{2\delta - \frac{\frr}{2}}.
$$
\end{myproof}

\section{Circle Method: Decomposition into Main Term and Error Term}
For the sake of exposition, we fix one $a \in \calA$ for the exponential sum $S_{N,a}(\theta)$ defined in \eqref{eq:RNdef},
$$
R_{N,a}(d) := \widehat{S}_{N,a}(d) = \int_0^1 S_{N,a}(\theta) e(-d \theta) d \theta =\sum_{\gamma \in \Omega_N} \textbf{1}_{\{ \langle \gamma e_2, \gamma_a e_2 \rangle = d \}}.
$$
Note that the proof for the rest of the chapters are the same for other exponential sums $S_{N}(\theta)$ and $S_{N,a}(\theta)$.
Our goal is to write the above integral as the sum of a main term and an error term,
$$
R_{N,a}(d) = \calM_{N,a}(d) + \calE_{N,a}(d).
$$
In fact, the main term is closely related to the integral of $S_{N,a}(\theta)$ over the following major arcs of level $\calQ$.
\begin{equation}\label{eq:majorarcs}
\frM_{\calQ} = \bigsqcup_{q<\calQ} \bigsqcup_{(a,q)=1} \left[  \frac{a}{q}-\frac{\calQ}{N} , \frac{a}{q} + \frac{\calQ}{N} \right],
\end{equation}
where $\calQ$ is defined in  \eqref{eq:Qdef}. 

More specifically, we construct a periodic test function with support on $\frM_{\calQ}$. Consider the triangle function $\psi(x)$ as in \eqref{eq:triangle}. We adjust the support around the origin, and obtain $\psi_N$ as follows.
\begin{equation}\label{eq:localtriangle}
\phi_N(x) := \psi\left(  \frac{N}{\calQ} x \right).
\end{equation}
Periodize $\psi_N$ to $\Psi_N$ on $\R /\Z$:
\begin{equation}\label{eq:localtrianglenew}
\Psi_N(\theta) := \sum_{m \in \Z} \phi_N( \theta +m ) ,
\end{equation}
and put each such spike at a major arc:
\begin{equation}\label{eq:localtrianglenew2}
\Psi_{\calQ,N}(\theta) = \sum_{q < \calQ} \sum_{(a,q)=1} \Psi_N\left(\theta - \frac{a}{q} \right)
\end{equation}
Hence, the support of $\Psi_{\calQ,N}$ is $\frM_{\calQ}$, and most of the mass centered at these major arcs.

Now, we define the main term as 
\begin{equation}\label{eq:majordef}
\calM_{N,a}(n) = \int_0^1 \Psi_{\calQ,N}(\theta) S_{N,a}(\theta) e(-n\theta) d \theta,
\end{equation}
while the error term is defined to be the complement of the main term,
\begin{equation}\label{eq:minordef}
\calE_{N,a}(n) := \int_0^1 (1-\Psi_{\calQ,N}) S_{N,a}(\theta) e(-n\theta) d \theta.
\end{equation}

\section{Main Term Analysis: Proof of Theorem \ref{thm:majorterm} }
In this section, we estimate the main term contribution, using the similar approach as the one in \cite{BK:2011}. First of all, the special set $\aleph$ constructed in $\S3.2$ allows us to split the exponential sum $S_{N,a}(\theta)$ into a product of modular and archimedean components in the major arcs.  

\vspace{2mm}
\subsection{Splitting $S_{N,a}(\theta)$ into Modular and Archimedean Components.}~

We first introduce the following Lemma which is crucial for us to analyze the main term.

\begin{lemma}\label{lemma:lm1}
Recall the existence of $\Omega_N$ in \eqref{eq:omegaNdefinition}. We write the set $\Omega_N$ as
\begin{equation}\label{eq:decom2}
\Omega_N = \Omega^{(1)} \cdot \aleph  \cdot  \Omega^{(2)},
\end{equation}
where $ \Omega^{(1)} = \Xi_{-J} \Xi_{-J+1} \cdots \Xi_{\bbj-1}$, and $ \Omega^{(2)} = \Xi_{\bbj+1} \Xi_{\bbj+2} \cdots \Xi_{J+1}.$

Then, for any $\gamma_2 \in \Omega^{(2)}$, we have
$$
\norm{\langle e_2, \vmi{\gamma_2}^{\perp} \rangle} \gg 1.
$$
The implied constant depends only on $\calA$.
\end{lemma}
\begin{proof}
We use the following observation.
\begin{equation*}
\begin{split}
\langle e_2, \vmi{\gamma_2}^{\perp} \rangle 
= & \langle e_2, \vpl{\gamma_2}  \rangle  \langle \vpl{\gamma_2} ,\vmi{\gamma_2}^{\perp} \rangle  \\
&+  \langle e_2, \vpl{\gamma_2}^{\perp}  \rangle  \langle \vpl{\gamma_2}^{\perp} ,\vmi{\gamma_2}^{\perp} \rangle .
\end{split}
\end{equation*}
Denote $ \langle e_2, \vpl{\gamma_2}  \rangle $ by $m_1$ and $ \langle \vpl{\gamma_2} ,\vmi{\gamma_2}^{\perp} \rangle$ by $m_2$.  Hence, \eqref{eq:innerproductbig} implies that
\begin{equation}
\norm{m_2} \geq \frac{1}{2}.
\end{equation}
On the other hand, from \eqref{eigenvectornearby} and \eqref{eq:vclose} , we obtain
\begin{equation}\label{eq:eq2}
m_1 = \langle  \frv , e_2 \rangle \left[  1+ O\left(\frac{1}{\calQ^5}\right) \right] > \frac{1}{\sqrt{4-2\sqrt{2}}}.
\end{equation}
Combining \eqref{eq:eq2} with the inequality 
$$\norm{\langle e_2, \vmi{\gamma_2}^{\perp} \rangle } \geq m_1\norm{m_2} - \sqrt{1-m_1^2}\sqrt{1-m_2^2} $$
 completes the proof.
\end{proof}

The following theorem is similar to Theorem 4.2 in \cite{BK:2011}, but we need to alter the proof since our $\Omega_N$ is different from the one in the original paper.

\begin{theorem}\label{thm:splitSn}
\emph{(\cite[p.~19]{BK:2011})}
Recall that $A =\max{\calA}$. Fixed some $a \in \calA$. There exists a function $\varpi_{N,a} : \R / \Z \rightarrow \C$, given explicitly in \eqref{eq:defofomega}, satisfying the following conditions.
\begin{enumerate}
\item The Fourier transform
\begin{equation*}\label{eq:varpidef}
\widehat{\varpi}_{N,a} : \Z \rightarrow \C : n \mapsto \int_0^1 \varpi_{N,a} (\theta) e(-n \theta) d \theta
\end{equation*}
is real-valued and non-negative, with 
\begin{equation}\label{eq:varpi0}
\varpi_{N,a}(0) = \sum_n \widehat{\varpi}_{N,a} (n) \ll \norm{\Omega_N}.
\end{equation}
\item For $\frac{1}{50A}N < n < \frac{1}{20}N$, we have
\begin{equation}\label{eq:varpihat}
\widehat{\varpi}_{N,a}(n) \gg \frac{\norm{\Omega_N}}{N}.
\end{equation}
\item With the exponential sum $S_{N,a}(\theta)$ defined in  \eqref{eq:exposumdef10}, on the major arcs $\theta = \frac{a}{q}+\beta \in \frM_{\calQ} $, we have
\begin{equation}\label{eq:Snsplit}
S_{N,a}\left( \frac{a}{q}+\beta   \right) = \nu_q(a) \varpi_{N,a}(\beta) + O\left(\frac{\norm{\Omega_N}}{\calQ^{4}}\right),
\end{equation}
where
\begin{equation}\label{eq:nu}
\nu_q(a) := \frac{1}{ \norm{\text{SL}_2(q)} } \sum_{\omega \in \text{SL}_2(q)} e \left(  \frac{a}{q} \langle   \omega e_2 ,e_2   \rangle \right).
\end{equation}
\end{enumerate}
\end{theorem}
\begin{remark}
Naturally, the local obstruction should appear in the modular part $ \nu_q(a)$. Therefore, in general, we have
$$
\nu_q(a) = \frac{1}{|S_{q}|} \sum_{\omega \in S_q} e \left(  \frac{a}{q} \langle   \omega e_2 ,e_2   \rangle \right).
$$
\end{remark}

\begin{proof}
Using the decomposition \eqref{eq:decom2}, we rewrite the exponenital sum $S_{N,a}(\theta)$ as
\begin{equation}\label{eq:Snsplit2}
S_N(\theta) = \sum_{\fra \in \aleph} \sum_{ \gamma_1 \in \Omega^{(1)} } \sum_{\gamma_2 \in \Omega^{(2)}}  e \left(    \theta \langle \gamma_1 \fra \gamma_2 e_2, \gamma_a e_2    \rangle          \right).
\end{equation}
From \eqref{eigenvectornearby}, \eqref{eq:alepheig}, and \eqref{eq:rangeofeig}, we see that for any $\fra \in \aleph$,
\begin{equation}\label{eq:closetovtwo}
\lambda(\fra) \asymp B, \quad
\norm{\vpl{\fra} -\frv } < \calQ^{-5}, \quad \text{and} \quad  \norm{\vpl{\gamma_2} -\frv} \ll \calQ^{-5}.
\end{equation}
Note that $\lambda(\fra)$ also satisfies nice archimedean property, \eqref{eq:alepheigv}. Hence, the next step is to convert the expression $\langle \gamma_1 \fra \gamma_2 e_2, \gamma_{a}e_2    \rangle   $ into one involving $\lambda(\fra)$. 

Write $v_{\pm}$ for $v_{\pm}(\fra \gamma_2)$, and observe that $ \langle  \fra \gamma_2 e_2, ^t\gamma_1\gamma_a e_2 \rangle = \langle \gamma_1 \fra \gamma_2 e_2,\gamma_a e_2    \rangle   $. We write $\fra \gamma_2 e_2$ as a linear combination of $v_{\pm}$, and obtain the following equation
\begin{equation}\label{eq:eq3}
\begin{split}
\langle  \fra \gamma_2 e_2, ^t\gamma_1e_2 \rangle 
& = \lambda(\fra \gamma_2) \frac{ \inner{e_2}{v_{-}^{\perp}} }{ \inner{ v_+}{ v_-^{\perp}} } \inner{v_+ }{^t\gamma_1\gamma_a e_2}  + \frac{1}{\lambda(\fra \gamma_2)} \frac{ \inner{e_2}{v_{+}^{\perp}} }{ \inner{ v_+}{ v_-^{\perp}} } \inner{v_- }{^t\gamma_1\gamma_a e_2} \\
& = \lambda(\fra) \lambda(\gamma_2) \frac{ \inner{e_2}{\vmi{\gamma_2}^{\perp}}  }{ \inner{ \frv }{ \vmi{\gamma_2}^{\perp} } } \inner{\frv }{^t\gamma_1\gamma_a e_2} \left[ 1+ O\left(\frac{1}{\calQ^5} \right)  \right],\end{split}
\end{equation}
where we used the construction of $\Omega_N$, \eqref{eigenvectornearby}, Lemma \ref{lemma:lm1}, and that
$$
v_- = v_- (\fra \gamma_2) = v_-(\gamma_2) (1+O(N^{-1})).
$$
Similarly, we have 
\begin{equation}\label{eq:agamma2}
\langle   \gamma_2 e_2, ^t\gamma_1\gamma_a e_2 \rangle  =\lambda(\gamma_2) \frac{ \inner{e_2}{\vmi{\gamma_2}^{\perp}}  }{ \inner{ \frv }{ \vmi{\gamma_2}^{\perp} } } \inner{\frv }{^t \gamma_1\gamma_a e_2} \left[ 1+ O\left(\frac{1}{\calQ^5} \right)  \right].
\end{equation}
Combining \eqref{eq:eq3} and \eqref{eq:agamma2}, we get
\begin{equation}\label{eq:fraexpression}
\begin{split}
\langle  \fra \gamma_2 e_2, ^t\gamma_1 \gamma_a e_2 \rangle  
&= \lambda(\fra) \langle   \gamma_2 e_2, ^t\gamma_1\gamma_a e_2 \rangle  \left[ 1+ O\left(\frac{1}{\calQ^5} \right)  \right], \\
& = \lambda(\fra) \langle   \gamma_2 e_2, ^t\gamma_1\gamma_a e_2 \rangle  + O\left(  N/\calQ^5  \right).
\end{split}
\end{equation}
Consequently, when $\theta$ is in the major arcs $\frM_{\calQ}$, the following equations hold.
\begin{equation}\label{eq:Snrealsplit}
\begin{split}
S_{N,a}\left(\frac{a}{q}+\beta \right) 
&= \sum_{\fra \in \aleph } \sum_{\gamma_1 \in \Omega^{(1)}} \sum_{\gamma_2 \in \Omega^{(2)}} e\left( \frac{a}{q} \langle  \fra \gamma_2 e_2, ^t\gamma_1\gamma_a e_2 \rangle  \right) e \left(  \beta \langle \fra  \gamma_2 e_2, ^t\gamma_1\gamma_a e_2 \rangle        \right)  ,\\
&= \sum_{\gamma_i \in \Omega^{(i)}}  \sum_{\omega \in \text{SL}_2(q) } 
e\left( \frac{a}{q} \langle \omega \gamma_2 e_2, ^t\gamma_1\gamma_a e_2 \rangle  \right) \sum_{\substack{ \fra \in \aleph \\ \fra \equiv \omega (\text{mod }q)  }} e \left(  \beta \langle \fra  \gamma_2 e_2, ^t\gamma_1\gamma_a e_2 \rangle        \right) , \\
&=  \sum_{\gamma_i \in \Omega^{(i)}}  \sum_{\omega \in \text{SL}_2(q) } 
e\left( \frac{a}{q} \langle \omega \gamma_2 e_2, ^t\gamma_1\gamma_a e_2 \rangle  \right) \sum_{u \in \calU} \sum_{\substack{ \fra \in \aleph_{u} \\ \fra \equiv \omega (\text{mod }q)  }} e \left(  \beta \langle \fra  \gamma_2 e_2, ^t\gamma_1\gamma_a e_2 \rangle        \right) ,\\
& =  \sum_{\gamma_i \in \Omega^{(i)}} \sum_{\omega \in \text{SL}_2(q)}  e\left( \frac{a}{q} \langle \omega \gamma_2 e_2, ^t\gamma_1\gamma_a e_2 \rangle  \right) \sum_{u \in \calU}  \sum_{\substack{ \fra \in \aleph_{u} \\ \fra \equiv \omega (\text{mod }q)  }}  e \left(  \beta u  \langle \gamma_2 e_2, ^t\gamma_1\gamma_a e_2 \rangle        \right)  +O\left(   \frac{\norm{\Omega_N}}{\calQ^{4} }   \right),  \\
& =  \sum_{\gamma_i \in \Omega^{(i)}} \sum_{\omega \in \text{SL}_2(q)}  e\left( \frac{a}{q} \langle \omega \gamma_2 e_2, ^t\gamma_1\gamma_a e_2 \rangle  \right) \sum_{u \in \calU}  \left(  \sum_{\substack{ \fra \in \aleph_u \\ \fra \equiv \omega (\text{mod }q)  }}  1 \right) e \left(  \beta u  \langle \gamma_2 e_2, ^t\gamma_1\gamma_a e_2 \rangle        \right)  +O\left(   \frac{\norm{\Omega_N}}{\calQ^{4} }   \right),  \\
\end{split}
\end{equation}
where $ \sum_{\gamma_i \in \Omega^{(i)}} $ is a double sum over $i = 1,2$, and the third equation comes from Taylor expansion, \eqref{eq:alepheigv}, and \eqref{eq:fraexpression}. 

To evaluate the term $  \sum_{\substack{ \fra \in \aleph_u \\ \fra \equiv \omega (\text{mod }q)  }}  1$, we need the distribution property of $\aleph$.
Specifically, by \eqref{eq:samecard} and \eqref{eq:equivaleph}, we have
\begin{equation}\label{eq:inner2}
 \sum_{\substack{ \fra \in \aleph_u \\ \fra \equiv \omega (\text{mod }q)  }}  1 = \frac{ \norm{\aleph_u} (1+ O(\calQ^{-4})) }{   \norm{\text{SL}_2(q)}  } = \frac{ \norm{\aleph} (1+ O(\calQ^{-4})) }{  \norm{\calU} \cdot \norm{\text{SL}_2(q)}  },
\end{equation}
where the implied constant is independent of  $u$, $\omega$, or $q$.
Equipped with \eqref{eq:inner2}, we get
\begin{equation}\label{eq:Snrealsplit2}
\begin{split}
S_{N,a}\left( \frac{a}{q} + \beta   \right) 
& = \frac{1}{ \norm{\text{SL}_2(q)} }  \sum_{\gamma_i \in \Omega^{(i)}}   \sum_{\omega \in \text{SL}_2(q)} e\left( \frac{a}{q} \langle \omega \gamma_2 e_2, ^t\gamma_1\gamma_a e_2 \rangle  \right) \left[ \frac{ \norm{\aleph}  }{ \norm{\calU}}   \sum_{u \in \calU } e \left(  \beta u \langle \gamma_2 e_2, ^t\gamma_1\gamma_a e_2 \rangle        \right) \right]+O\left(   \frac{\norm{\Omega_N}}{\calQ^{4} }   \right) ,\\
& = \frac{1}{ \norm{\text{SL}_2(q)} } \sum_{\omega \in \text{SL}_2(q)}   e\left( \frac{a}{q} \langle \omega e_2,  e_2 \rangle  \right) \left[ \frac{ \norm{\aleph}  }{ \norm{\calU}}  \sum_{\gamma_i \in \Omega^{(i)}}   \sum_{u \in \calU } e \left(  \beta u \langle \gamma_2 e_2, ^t\gamma_1\gamma_a e_2 \rangle        \right) \right]+O\left(   \frac{\norm{\Omega_N}}{\calQ^{4} }   \right), \\
\end{split}
\end{equation}
where the second equation holds since for each fixed $\gamma_1,$ $\gamma_2$, the $\omega $ sum runs through all of $\text{SL}_2(q)$. One can see that we already acquired the first term $\nu_q(a)$ in the above expression.

Next, we want to understand the distribution of frequencies for $ \sum_{\gamma_i \in \Omega^{(i)}}   \sum_{u \in \calU } e \left(  \beta u \langle \gamma_2 e_2, ^t\gamma_1\gamma_a e_2 \rangle        \right) $, and hence, we approximate 
$u \langle \gamma_2 e_2, ^t \gamma_1\gamma_a e_2 \rangle $ by nearby integers.

Fix $\gamma_1 \in  \Omega^{(1)}$, $\gamma_2 \in  \Omega^{(2)}$, and $u \in \calU$. For any integer $m $ close to $u \langle \gamma_2 e_2, ^t \gamma_1\gamma_a e_2 \rangle $,
\begin{equation}\label{eq:minteger}
\norm{ m - u \langle \gamma_2 e_2, ^t \gamma_1\gamma_a e_2 \rangle } \leq B  \langle \gamma_2 e_2, ^t \gamma_1 \gamma_a e_2 \rangle / \calQ^{5},
\end{equation}
we have 
\begin{equation}\label{eq:resultofm}
e \left(  \beta u \langle \gamma_2 e_2, ^t\gamma_1\gamma_a e_2 \rangle    \right) = e \left(  \beta m \right) (1+O(\calQ^{-4})).
\end{equation}
In addition, the number of integers $m$ in \eqref{eq:minteger} is 
$$2B \langle \gamma_2e_2, ^t \gamma_1\gamma_a e_2 \rangle / \calQ^5 + O(1)  =   (1+O(\calQ^{-4})) 2B \langle \gamma_2e_2, ^t \gamma_1\gamma_a e_2 \rangle / \calQ^5.$$ 
Thus, 
\begin{equation}\label{eq:Snrealsplit3}
e \left(  \beta u \langle \gamma_2 e_2, ^t\gamma_1\gamma_a e_2 \rangle    \right) =  \frac{\calQ^5(1+O(\calQ^{-4}))}{2B\langle \gamma_2e_2, ^t \gamma_1\gamma_a e_2 \rangle  } \sum_{m \in \Z \text{ , } \norm{ \frac{m}{\langle \gamma_2e_2, ^t \gamma_1 \gamma_a  e_2 \rangle } -u }\leq \frac{B}{\calQ^5} } e(\beta m) ,
\end{equation}
for $N$ large enough.

Rearranging the sum over $u$ and $m$, and  inserting \eqref{eq:Snrealsplit3} into \eqref{eq:Snrealsplit2} leads to
\begin{equation}\label{eq:Snrealsplit8}
S_{N,a}\left(  \frac{a}{q} +\beta   \right) = \nu_q(a) \varpi_{N,a}(\beta) (1+O(\calQ^{-4}))+ O\left(\frac{\norm{\Omega_N}}{\calQ^{4}}\right).
\end{equation}
Here, the term $\varpi_N(\beta) $ is defined as
\begin{equation}\label{eq:defofomega}
\varpi_{N,a}(\beta) :=  \frac{ \norm{\aleph}  }{ \norm{\calU}}  \sum_{\gamma_i \in \Omega^{(i)}}  \frac{\calQ^5}{2B\langle \gamma_2e_2, ^t \gamma_1\gamma_a e_2 \rangle  } \sum_{m\in \Z} e(\beta m) \sum_{u \in \calU} \textbf{1}_{\norm{ \frac{m}{\langle \gamma_2e_2, ^t \gamma_1\gamma_a e_2 \rangle } -u }\leq \frac{B}{\calQ^5}}.
\end{equation}
Notice from \eqref{eq:commondiff} that the possible range of $m$ is of size $B\langle \gamma_2e_2, ^t \gamma_1\gamma_a e_2 \rangle $. Therefore, $\norm{\varpi_{N,a}(\beta)} $  is bounded by
\begin{equation}\label{eq:major3}
\norm{\varpi_{N,a}(\beta)} \ll \frac{ \norm{\aleph} }{ \norm{\calU} } \# \Omega^{(1)} \# \Omega^{(2)} \calQ^5 \ll \# \Omega_N,
\end{equation}
and we may rewrite \eqref{eq:Snrealsplit8} as
\begin{equation}\label{eq:Snrealsplit4}
S_{N,a}\left(  \frac{a}{q} +\beta   \right) = \nu_q(a) \varpi_{N,a}(\beta) + O\left(\frac{\norm{\Omega_N}}{\calQ^{4}}\right).
\end{equation}
Moreover, one can easily check that
\begin{equation}\label{eq:varpifouriertransform}
\widehat{\varpi}_{N,a}(n) =  \frac{ \norm{\aleph}  }{ \norm{\calU}}  \sum_{\gamma_i \in \Omega^{(i)}}  \frac{\calQ^5}{2B\langle \gamma_2e_2, ^t \gamma_1\gamma_a e_2 \rangle  } \sum_{u \in \calU} \textbf{1}_{\norm{ \frac{n}{\langle \gamma_2e_2, ^t \gamma_1\gamma_a e_2 \rangle } -u }\leq \frac{B}{\calQ^5}}
\end{equation}
is real and non-negative. 

We still need to show that $ n/\langle \gamma_2e_2, ^t \gamma_1\gamma_a e_2 \rangle$ is comparable to $N/B$. First of all, using \eqref{comtracenorm}, \eqref{normsupsecond}, and Lemma \ref{lemma:controloneigenvalues}, we get
\begin{equation}
\frac{1}{16} \frac{N}{B} < \langle \gamma_2 e_2, ^t \gamma_1\gamma_a e_2 \rangle = \langle \gamma_1 \gamma_2 e_2,\gamma_a e_2 \rangle < 9A \frac{N}{B}.
\end{equation}
Since $\frac{1}{50} N < n < \frac{1}{20}N$, we get 
\begin{equation}\label{eq:inequalityofsize}
\frac{1}{450A} B \leq \frac{n}{ \langle \gamma_2 e_2, ^t \gamma_1\gamma_a e_2 \rangle } < \frac{399}{400} B.
\end{equation}
Finally, by the definition \eqref{eq:commondiff} of $u \in \calU$ in this range, the innermost sum in \eqref{eq:varpifouriertransform} is not vacuous.
$$\sum_{u \in \calU} \textbf{1}_{\norm{ \frac{n}{\langle \gamma_2e_2, ^t \gamma_1\gamma_a e_2 \rangle } -u }\leq \frac{B}{\calQ^5}} \geq 1.$$ 
Therefore, $\widehat{\varpi}_{N,a}(n) $ is bounded  below by
\begin{equation}\label{eq:lowerboundvarpi}
\widehat{\varpi}_{N,a}(n) \gg \frac{\norm{\aleph}}{\norm{\calU}} \sum_{\gamma_i \in \Omega^{(i)}} \frac{\calQ^5}{2B\langle \gamma_2e_2, ^t \gamma_1\gamma_a e_2 \rangle  } \gg \frac{\norm{\aleph} \norm{\Omega^{(1)}} \norm{\Omega^{(2)}}}{N} = \frac{\norm{\Omega_N}}{N}.
\end{equation}
\end{proof}
\vspace{2mm}
\subsection{The Lower Bound of  The Main Term}~

The next lemma gives a nice elementary result of the Ramanujan sum, and is crucial in estimating contribution of the main term.

\begin{lemma}\label{lemma:lmajor1}
Assume that $q$ is a prime power $p^t$. Let $c_q(m)$ be the Ramanujan's sum defined as follows.
$$
c_q(m) = \sum_{(a,q)=1} e \left( \frac{a}{q} m  \right).
$$
Recall the admissible set $\cS_{q}$ defined in Remark \ref{r:111}. We consider the function $C_q(n)$ which averages $c_q(m)$ over the group $\text{SL}_2(q)$.
\begin{equation}\label{eq:major8}
C_{q}(n) = \frac{1}{ \norm{\cS_q }}  \sum_{\omega \in \cS_q }  c_q( d-n ),
\end{equation}
where $\omega = 
\begin{pmatrix} 
a & b \\ c & d
\end{pmatrix}.
 $
Let $\calB$ be the bad modulus in Remark \ref{remark:1.1}. Then we have
\begin{equation}\label{eq:major9}
C_q(n)   =    \begin{dcases*}
        \frac{-1}{p+1},  & when $t =1$, $p \mid n$, and $p \nmid \calB$.\\
        \frac{1}{p^2-1}, &  when $t=1$, $p \nmid n$, and $p \nmid \calB$. \\
        0, &   when $t \geq 2$ and $p \nmid \calB$. \\
       \frac{p^t \norm{\left\{  \omega \in \cS_q : \omega \equiv n (p^t) \right\}} - p^{t-1} \norm{\left\{  \omega \in \cS_q : \omega \equiv n (p^{t-1}) \right\}} }{\cS_q}, & when $p^t \mid \calB$. \\
       0, & when $p \mid \calB$ and $p^t \nmid \calB$.
        \end{dcases*}
\end{equation}
\end{lemma}
\begin{proof}
We only prove the case when $p \nmid \calB$. The similar proof goes for $p \mid \calB$. 

First of all, by Mobius inversion formula, we have
\begin{equation}\label{eq:major7}
c_q(m) = \sum_{s \mid (q,m) } s\mu \left(  \frac{q}{s} \right), \text{ $\mu$ is the Mobius function.}
\end{equation}
Using \eqref{eq:major7} and the fact that $\norm{\cS_q}  = p^{3t-2} (p^2-1)$, we can easily prove the two cases when $t = 1$. For $t \geq 2$, by \eqref{eq:major7}, we rewrite $C_q(n)$ as 
\begin{equation}\label{eq:major10}
\begin{split}
C_{q}(n) 
&=  \frac{1}{ \norm{\text{SL}_2(q)} }  \sum_{\omega \in \text{SL}_2(q)}   \sum_{s\mid (q,m) } s \mu \left(  \frac{q}{s} \right) \text{, where } m =   \langle \omega e_2 , e_2 \rangle - n, \\     
& = \left( p^{t} - p^{t-1} \right) \sum_{\omega \in  \text{SL}_2(q)} \textbf{1}_{ \{ d \equiv n (p^t)  \} } - p^{t-1} \sum_{\omega \in  \text{SL}_2(q)} \textbf{1}_{ \left\{ \substack{d \equiv n (p^{t-1})\\ d \not\equiv n (p^t)   }  \right\} }.
\end{split}
\end{equation}
We will show that 
\begin{equation}\label{eq:major11}
(p-1) \sum_{\omega \in  \text{SL}_2(q)} \textbf{1}_{ \{ d \equiv n (p^t)  \} }  =\sum_{\omega \in  \text{SL}_2(q)} \textbf{1}_{ \left\{ \substack{d \equiv n (p^{t-1})\\ d \not\equiv n (p^t)   }  \right\} }.
\end{equation}
In fact, \eqref{eq:major11} is true if we can show that for any $\gamma \equiv
\begin{pmatrix} 
\ast & \ast \\ \ast & n 
\end{pmatrix} \pmod{p^{t-1}} ,$ we have
\begin{equation}\label{eq:major12}
(p-1) \sum_{\substack{ \omega \in  \text{SL}_2(q) \\ \omega \equiv \gamma (p^{t-1}) } } \textbf{1}_{ \{ d \equiv n (p^t)  \} }  =\sum_{\substack{ \omega \in  \text{SL}_2(q) \\ \omega \equiv \gamma (p^{t-1}) } } \textbf{1}_{ \left\{ d \not\equiv n (p^t)     \right\} }.
\end{equation}
Indeed, for any $\omega \equiv \begin{pmatrix} 
a_1 & b_1 \\ c_1 & n 
\end{pmatrix} \pmod{p^{t-1}}$, the general expression for  $\omega$ is 
\begin{equation}\label{eq:major13}
\omega =\begin{pmatrix} 
a & b \\ c & d
\end{pmatrix} =   \begin{pmatrix} 
a_1+p^{t-1}k_1 & b_1+p^{t-1}k_2 \\ c_1+p^{t-1}k_3 & n+p^{t-1}k_4, 
\end{pmatrix}
\end{equation}
where $0 \leq k_i < p $. and satisfy
\begin{equation}\label{eq:major14}
p \mid k_1 n + k_4 a_1 - k_3 b_1 - k_2 c_1 + \frac{a_1n-b_1c_1}{p^{t-1}}.
\end{equation}

Now, one can see that no matter what choice of $k_4$ we fix, there are always $p^2$ choices for the tuple $(k_1,k_2,k_3)$. Consequently, we must have \eqref{eq:major12} which concludes \eqref{eq:major9}.
\end{proof}

Equipped with Theorem \ref{thm:splitSn}, we now restate and give the proof of  Theorem \ref{eq:majorcontribution111} as in \cite{BK:2011} with a different range of $n$.
 
\begin{theorem}\label{thm:majorcontribution} 
\emph{(\cite[p.~23]{BK:2011})}
For any $n$ admissible to the set $\Gamma_{\calA, a}$, and that $\frac{1}{40}N \leq n < \frac{1}{25}N$, we have
\begin{equation}\label{eq:majorcontribution}
\calM_{N,a}(n) \gg \frac{1}{\log \log N} \frac{\# \Omega_N}{N}.
\end{equation}
\end{theorem}

\begin{proof}
The proof will go as assuming the reduction of $\Gamma$ under all $q \geq 1$ is full. For the general case, see Remark \ref{remark:4.5}  We continue to work on \eqref{eq:majordef}. By the definition of   $\Psi_{\calQ,N}(\theta) $, and \eqref{eq:Snrealsplit4}, we can write $\calM_{N,a}(n) $ as
\begin{equation}\label{eq:major1}
\begin{split}
\calM_{N,a}(n) &= \int_0^1  \sum_{q < \calQ} \sum_{(a,q)=1} \Psi_N\left(\theta - \frac{a}{q} \right) \left[ \nu_q(a) \varpi_{N,a}(\beta)+ O\left(\frac{\norm{\Omega_N}}{\calQ^{4}} \right) \right] e(-n \theta) d \theta, \\
&= \int_0^1  \sum_{q < \calQ} \sum_{(a,q)=1} \Psi_N\left(\theta - \frac{a}{q} \right) \nu_q(a) \varpi_{N,a}(\beta)e(-n \theta) d \theta + O \left( \calQ^2\cdot \frac{\calQ}{N}  \cdot \frac{\norm{ \Omega_N}}{\calQ^{4} }  \right). \\
\end{split}
\end{equation}

\noindent
Replacing $\theta$ by $\frac{a}{q}+ \beta$, we get
\begin{equation}\label{eq:major2}
\begin{split}
\calM_{N,a}(n) &= \sum_{q < \calQ} \sum_{(a,q)=1}  \nu_q(a) e\left( -n\frac{a}{q}  \right)  \int_0^1  \Psi_N\left(\beta \right)\varpi_{N,a}(\beta)  e(-n \beta)  d \beta + O \left( \frac{  \norm{ \Omega_N} }{\calQ N}    \right), \\
&=  \mathfrak{G}_{\calQ}(n)  \frT_{N,a}(n) + O \left( \frac{  \norm{ \Omega_N }}{\calQ N}    \right), \\
\end{split}
\end{equation}
where
\begin{equation}\label{eq:major5}
 \mathfrak{G}_{\calQ}(n) =  \sum_{q < \calQ} \sum_{(a,q)=1}  \nu_q(a) e\left( -n\frac{a}{q}  \right),  \quad \frT_{N,a}(n) =  \int_0^1  \Psi_N\left(\beta \right)\varpi_{N,a}(\beta)  e(-n \beta)  d \theta ,
\end{equation}
and the implied constant depends only on $\calA$.

Our first task is to estimate the singular series $\frG_{\calQ}(n)$. We write the singular series as follows.
\begin{equation}\label{eq:major6}
\frG_{\calQ}(n) = \sum_{q < \calQ} \frac{1}{\norm{ \text{SL}_2(q) }} \sum_{ \omega \in  \text{SL}_2(q)  } \sum_{(a,q) = 1} e\left( \frac{a}{q} \left(  \langle \omega e_2 , e_2 \rangle - n    \right)   \right) =  \sum_{q < \calQ} \frac{1}{\norm{ \text{SL}_2(q) }} \sum_{ \omega \in  \text{SL}_2(q)  } c_{q}\left(  \langle \omega e_2 , e_2 \rangle - n \right),
\end{equation}
where $c_q(m)$ is the Ramanujan sum, see Lemma \ref{lemma:lmajor1}. 

Recall the average function $C_q(n)$ in Lemma \ref{lemma:lmajor1}. Since the Ramanujan's sum is multiplicative, by Chinese remainder theorem, $C_q(n)$ is also multiplicative. Consider the following indicator function
$$
\mathfrak{p}_n (q) = \begin{dcases*}
       1, &  if $q = p_1 \cdots p_{2k}$, for distinct primes $p_i$, and for all $i$, $p_i \mid n$. \\
        0, &  otherwise.
        \end{dcases*}
$$
Then the contribution of $\sum_{q > \calQ} C_q(n) $ is at most
\begin{equation}\label{eq:major15}
\sum_{q > \calQ} C_q(n) \leq  \left( \sum_{q > \calQ} \mathfrak{p}_n (q) \frac{1}{q}  \right)\prod_{p \nmid n} \left( 1+ \frac{1}{p^2-1}  \right)  \leq \frac{2}{\calQ} \sum_{q > \calQ} \mathfrak{p}_n (q) = o\left( \frac{1}{\log \log n}   \right),
\end{equation}
where the last equality comes from the definition of $\calQ$ and the assumption that $\frac{1}{40}N \leq n < \frac{1}{25}N$. 

Therefore, we can extend the sum $\sum_{q < \calQ}$ in $\frG_{\calQ}(n)$ to $\sum_{q < \infty}$ which leads us to the following new series
\begin{equation}\label{eq:major16}
\frG(n) = \sum_{q < \infty } C_q(n).
\end{equation}
Moreover, by the multiplicativity of $C_q(n)$ and Lemma \ref{lemma:lmajor1}, we have
\begin{equation}
\frG(n) = \prod_{p \nmid n} \left(  1+ \frac{1}{p^2-1} \right) \cdot \prod_{p \mid n}   \left( 1- \frac{1}{p+1}  \right) \gg  \prod_{p \mid n}   \left( 1- \frac{1}{p+1}  \right).
\end{equation}
To estimate $ \prod_{p \mid n}   \left( 1- \frac{1}{p+1}  \right)$, we see that 
\begin{equation}\label{eq:major17}
\begin{split}
{\left[       \prod_{p \mid n}   \left( 1- \frac{1}{p+1}  \right)     \right] }^{-1} = \prod_{p \mid n} \left(  1+ \frac{1}{p} \right) \leq \frac{1}{n} \sigma_1(n) \ll \log \log n,
\end{split}
\end{equation}
where $\sigma_1(n) = \sum_{d \mid n} d$, and the last inequality comes from Robin's inequality \cite{Rob:84}. Thus, we get
$$
\frG(n) \gg \frac{1}{\log \log n},
$$
and by \eqref{eq:major15}, we obtain 
\begin{equation}\label{eq:major18}
\frG_{\calQ}(n) \gg \frac{1}{\log \log n}.
\end{equation}

Let us move on to the singular integral $ \frT_N(n)$. Elementary Fourier analysis then tells us that
\begin{equation}\label{eq:major19}
 \mathfrak{T}_{N,a}(n) = \sum_{m \in \Z}   \widehat{\phi}_N(n-m) \widehat{\varpi}_{N,a}(m)=\frac{\calQ}{N} \sum_{m \in \Z}   \widehat{\psi}\left( \frac{\calQ}{N}(n-m)  \right) \widehat{\varpi}_{N,a}(m).
\end{equation}
Moreover, by the definition of $\psi$ in \eqref{eq:triangle}, we have
$$ 
\widehat{\psi}(x) > 2/5,
$$ 
when $\norm{x} < 1/2 $. Consequently, 
\begin{equation}\label{eq:major20}
\frT_N(n) \gg \frac{\calQ}{N} \sum_{ \norm{m-n} < N/(2\calQ) } \widehat{\varpi}_{N,a}(m).
\end{equation}
Since $\frac{1}{40}N \leq n < \frac{1}{25}N$, we can ensure that for $N$ large enough, the following inequality is true for any $m$ with $\norm{m-n} < N/(2\calQ)$.
$$
\frac{1}{50}N < m < \frac{1}{20}N.
$$
Therefore, we may apply Theorem \ref{thm:splitSn} and get
\begin{equation}\label{eq:major21}
\frT_{N,a}(n) \gg \frac{\calQ}{N} \cdot \frac{N}{2\calQ} \cdot \frac{\# \Omega_N}{N} \gg \frac{\# \Omega_N}{N}.
\end{equation}
Combining \eqref{eq:major21} with \eqref{eq:major18}, we obtain
\begin{equation}
\calM_{N,a}(n) \gg \frac{1}{\log \log N} \frac{\# \Omega_N}{N},
\end{equation}
where we also use the fact that $\calQ^{-1}= o \left( \frac{1}{\log \log N}  \right)$. 
\end{proof}

\begin{remark}\label{remark:4.5}
In general, when $n$ is admissible, we have
\begin{equation}
\frG(n) = \prod_{\substack{p \nmid n \\ p \nmid \calB}} \left(  1+ \frac{1}{p^2-1} \right) \cdot \prod_{\substack{p \mid n \\ p \nmid \calB}}   \left( 1- \frac{1}{p+1}  \right) \cdot \prod_{p \mid \calB} \left(  1+ C_p(n) + C_{p^2}(n) + ... + C_{p^t}(n) \right),
\end{equation}
where $p^t || \calB$. By the fact that every term in the product $\prod_{p \mid \calB} $ is positive, we still obtain the same lower bound
$$
\frG(n)  \gg \frac{1}{\log \log N}.
$$
However, when $n $ is not admissible, from Lemma \ref{lemma:lmajor1}, one of the terms in the product $\prod_{p \mid \calB} $  must be zero. Hence, we no longer have \eqref{eq:majorcontribution}.
\end{remark}

\begin{remark}
As mentioned before, the above theorem holds for all $\calM_{N,a}(n) $ and $\calM_{N}(n) $.
\end{remark}

\section{Setup for Error Term Analysis}
Recall that after fixing an integer $a \in \calA$, our task is to evaluate the following error term, see \eqref{eq:minordef}.
\begin{equation*}
\calE_{N,a}(n) := \int_0^1 (1-\Psi_{\calQ,N}) S_{N,a}(\theta)e(-n \theta) d\theta.
\end{equation*}
As it stands for the error term, we hope to bound the $L^2$-norm of $ (1-\Psi_{\calQ,N}) S_{N,a}(\theta) e(-n \theta)$. More specifically, we achieve this through bounding $\int_0^1\norm{S_{N,a}(\theta)}^2 d \theta$.

First of all, we decompose $[0,1]$ into dyadic regions. Set 
$$N' = 2^{26} (A+1)^2N,$$
where $A = \max{\calA}$. Dirichlet approximation states that for every number $\theta \in [0,1]$, there exists a fraction $s/q$ with $1\leq q \leq \sqrt{N'}$ for which
\begin{equation}\label{eq:dirichlet}
\norm{\theta - \frac{s}{q}} < \frac{1}{q \sqrt{N'}}.
\end{equation}
Consequently, we have
\begin{equation}\label{eq:dyadic}
\int_0^1\norm{S_{N,a}(\theta)}^2 d \theta  \ll  \sum_{ \substack{1 \leq Q < {N'}^{1/2} \\ \text{dyadic}}} \sum_{\substack{\widetilde{C} \leq K<\frac{{N'}^{1/2}}{Q} \\ \text{dyadic}  } }\int_{W_{Q,K}} \norm{S_{N,a}(\theta)}^2 d \theta.
\end{equation}
where the regions $W_{Q,K}$ are defined as follows.

Recall the constant $\widetilde{C}$ defined in \eqref{eq:wideC}. For $K \geq 2\widetilde{C}$, we set
\begin{equation}\label{eq:Wqk2}
W_{Q,K} = \left\{  \theta = \frac{s}{q} + \frac{\frl}{ \calT N'} + \mathfrak{t} :  \frac{1}{2}Q \leq q < Q \text{, } (s,q) = 1 \text{, } \frac{K \calT}{2} \leq  \norm{\frl} < K \calT \text{, }  \norm{\frt} < \frac{1}{\calT N'}  \right\},
\end{equation}
where $q,a,\frl \in \Z$ and $\calT = \calT(Q,K)$ is some parameter we will set later.  On the other hand, for each $Q$ in the summation, we have one choice of $K$ with $\widetilde{C} \leq  K < 2 \widetilde{C} $. In this case, we set 
\begin{equation}\label{eq:Wqk10}
W_{Q,K} = \left\{  \theta = \frac{s}{q} +\beta :  \frac{1}{2}Q \leq q < Q \text{, } (s,q) = 1\text{, }  \norm{\beta} < \frac{K}{N'}  \right\}.
\end{equation}
It is not hard to show that the union of these regions $W_{Q,K}$ is the whole interval $[0,1]$. Note that in the construction of $S_{N,a}(\theta)$, we still use the original $N$ instead of $N'$.

We now focus on the case when $K \geq 2 \widetilde{C}$. For any $\theta \in W_{Q,K}$ and $\gamma \in \Omega_N$, we have by Taylor expansion and Lemma \ref{lemma:controloneigenvalues} that
\begin{equation}\label{SN2}
\norm{e(\langle \gamma e_2,\gamma_a e_2 \rangle \theta ) - e\left(\langle \gamma e_2,\gamma_a e_2 \rangle \left( \frac{s}{q} + \frac{\frl}{\calT N'} \right) \right)} \leq \frac{1}{\calT}.
\end{equation}
Thus,
\begin{equation}\label{SN3}
\begin{split}
& \norm{S_{N,a}(\theta) - S_{N,a}\left(\frac{s}{q} + \frac{\frl}{\calT N'}\right)}  \leq \frac{\# \Omega_N}{\calT} .\\
\end{split}
\end{equation}
By triangle inequality, we get
$$
\norm{S_{N,a}(\theta )}^2 \leq 2\left( \norm{S_{N,a}\left(\frac{s}{q} + \frac{\frl}{\calT N}\right)}^2 + \frac{{ \left(\# \Omega_N \right) }^2}{\calT^2} \right) .
$$
Finally, integrating $\norm{S_{N,a}(\theta )}^2 $ over $W_{Q,K}$, we obtain
\begin{equation}\label{eq:SN10}
 \int_{W_{Q,K}} \norm{S_{N,a}(\theta)}^2 \ll \frac{1}{\calT N} \sum_{\frac{1}{2}Q \leq q < Q} \sum_{ (s,q) = 1} \sum_{\frac{K \calT}{2} \leq  \norm{\frl} < K \calT } \norm{S_{N,a}\left(\frac{s}{q} + \frac{\frl}{\calT N}\right)}^2 + \frac{KQ^2}{\calT^2N} {\left(\# \Omega_N\right)}^2,
\end{equation}
where the implied constant is absolute. Taking
$$\calT = KQ^{\frac{3}{2}},$$
\eqref{eq:SN10} now becomes
\begin{equation}\label{SN4}
\int_{W_{Q,K}} \norm{S_{N,a}(\theta)}^2 \ll \frac{1}{\calT N} \sum_{\theta \in P_{Q,K} } \norm{S_{N,a}\left(\theta \right)}^2 + \frac{1}{KQ} \frac{{\left(\# \Omega_N\right)}^2}{N},
\end{equation}
where 
$$
P_{Q,K} = \left\{  \frac{s}{q} + \frac{\frl}{\calT N}  :  \frac{1}{2}Q \leq q< Q\text{, }  q\geq s\text{, } (q,s) = 1\text{, }\frac{1}{2}K \calT \leq \frl < K \calT \right\}.
$$
\begin{remark}
The inequality \eqref{eq:SN10} can be treated as approximating the integral with Riemann sum. The larger $T$ we pick, the closer the triple sum in \eqref{eq:SN10} is to the integral $ \int_{W_{Q,K}} \norm{S_{N,a}(\theta)}^2$. We specifically choose $\calT =  KQ^{\frac{3}{2}}$ to have a $KQ $ saving in the error term $\frac{KQ^2}{\calT^2N} {\left(\# \Omega_N\right)}^2$. This concept can also be found in \cite{FK:2013}.
\end{remark}

On the other hand, when $\widetilde{C} \leq K < 2 \widetilde{C}$, we use the following inequality instead.
\begin{equation}\label{SN41}
\int_{W_{Q,K}} \norm{S_{N,a}(\theta)}^2 \ll \frac{K}{ N}  \sup_{\norm{\beta} < K/N}\sum_{P_{Q,\beta} } \norm{S_{N,a}\left(\frac{s}{q}+ \beta \right)}^2 \ll \frac{1}{ N}  \sup_{\norm{\beta} < K/N}\sum_{P_{Q,\beta} } \norm{S_{N,a}\left(\frac{s}{q}+ \beta \right)}^2,
\end{equation}
where 
$$
P_{Q,\beta} = \left\{  \frac{a}{q} +\beta  :  \frac{1}{2}Q \leq q< Q\text{, }  q\geq s\text{, } (q,s) = 1 \right\}.
$$

 In \S6 and \S7, we derive bounds for $\sum_{P_{Q,K} } \norm{S_{N,a}\left(\frac{s}{q} + \frac{\frl}{\calT N}\right)}^2$for different magnitudes of $K$ and $Q$. When $K$ is of constant magnitude $\widetilde{C}$, one can slightly altered the proof in $\S6$ and $\S7$ to bound the sum $\sum_{P_{Q,\beta} } \norm{S_{N,a}\left(\frac{s}{q} +\beta\right)}^2$.  Keep in mind that for each region $W_{Q,K}$, we have the following bounds for the parameters $Q$ and $K$.
$$
Q < 2^{13}(A+1)\sqrt{N} \text{, and }
KQ < 2^{13}(A+1)\sqrt{N}.
$$

\section{Error Term Analysis: Large $KQ$}

We explain in detail the case when $K \geq 2 \widetilde{C}$.  For the case $\widetilde{C} \geq K < 2 \widetilde{C}$, the proof needs minor changes, see Remark \ref{remark:7.1} and \ref{remark:7.2}.

Recall that our task is to estimate the exponential sum
\begin{equation}\label{exponentialsumexpression}
\sum_{P_{Q,K}} \norm{S_{N,a}(\theta)}^2,
\end{equation}
where $P_{Q,K} = \left\{  \frac{s}{q} + \frac{\frl}{\calT N'}  :  \frac{1}{2}Q \leq q< Q\text{, }  q\geq s\text{, } (q,s) = 1\text{, }\frac{1}{2}K \calT \leq \frl < K \calT \right\}$, $N' = 2^{26}(A+1)^2N$, and $\calT=KQ^{\frac{3}{2}}$. Hence, the size of $P_{Q,K}$ is $KQ^2\calT$. 

Our goal is to have a bound slightly better than $\frac{ \ppowtwo{\# \MOmN} }{N}$. Specifically, we want some extra saving of $K$ or $Q$ as follows.
$$
\frac{1}{\calT N}\sum_{P_{Q,K}} \norm{S_{N,a}(\theta)}^2 \ll \frac{ \ppowtwo{\# \MOmN} }{N} \frac{1}{K^{c_0}Q^{c_1}}.
$$
In this section, we use a ``triple'' Kloosterman refinement to give a bound for $\sum_{P_{Q,K}} \norm{S_{N,a}(\theta)}^2.$ This bound will suffice as long as $KQ$ is large. 

\begin{remark}
The reason we call the method we use in this section a ``triple'' Kloosterman refinement is that $P_{Q,K}$ takes summation over $s$, $q$, and $\ell$. Bringing in a new sum allows us to gain extra cancellations in the exponential sum.The concept of ``triple'' Kloosterman refinement originated in \cite{Kob:92}.
\end{remark}
First of all, we recall an observation in \cite{FK:2013}. A similar statement can also be found in \cite{Kon:02}.

\begin{lemma}\label{trickCZ}
\emph{(\cite[p.~33]{FK:2013})}
Let $W$ be a finite subset of $[0,1]$ such that $|W| >3.$ Suppose that $f:W \rightarrow \R_+ \cup  {\{0\}}$ is a non-negative function such that for any subset $Z \subset  W$, we have
\begin{equation}
\sum_{\theta \in Z} f(\theta) \leq C_1 {|Z|}^{\frac{1}{2}},
\end{equation}
where $C_1$ is independent of the choice of $Z$. Then we have
\begin{equation}
\sum_{\theta \in W } f^2(\theta) \leq 2 C_1^2 \log{|W|}.
\end{equation}
\end{lemma}

\begin{remark}
As we shall see later, this Lemma contribute more saving than the following approach given that our $L^{\infty}$ bound saves only $KQ$.
$$
\sum_{\theta \in W } f^2(\theta)  \leq \left(\sup_{\theta \in W}{f} \right) \sum_{\theta \in W }f(\theta).
$$
\end{remark}

Hence, by Lemma (\ref{trickCZ}), it is natural to look for a universal bound for $\sum_{Z} \norm{S_{N,a}(\theta)},$ where $Z$ is an arbitrary subset of $P_{Q,K}$. Specifically, the following theorem holds for any subset $Z \subset P_{Q,K}$.
\begin{theorem}\label{Tfirstminorarc}
Assume that $K \geq  2 \widetilde{C}$. For any subset $Z \subset P_{Q,K}$, we have 
\begin{equation}\label{firstminorarc}
 \sum_{\theta \in Z} \norm{S_{N,a}(\theta)} \ll \#\MOmN \cdot \frac{  N^{1-\delta + \frac{\frr}{2}} }{\sqrt{KQ}} \cdot  {|Z|}^{\frac{1}{2}} \cdot \calT^{\frac{1}{2}}.
\end{equation}
\end{theorem}

\begin{remark}
Throughout the rest of the paper, the implied constants for inequalities with symbol $\ll$ depen only on $\calA$ and $\epcon$. Notice that $\calA$ and $\epcon$ are fixed in the beginning.
\end{remark}

\begin{remark}
When $K$ is at constant level $\widetilde{C}$, Theorem \ref{Tfirstminorarc} read as follows instead. For any subset $Z \subset P_{Q,\beta}$, we have 
$$
 \sum_{\theta \in Z} \norm{S_{N,a}(\theta)} \ll \#\MOmN \cdot \frac{  N^{1-\delta} 2^{c\ppowtwo{\log \log N} } }{\sqrt{Q}} \cdot  {|Z|}^{\frac{1}{2}} .
$$

\end{remark}

\begin{proof}
We decompose $\MOmN$ as follows.
\begin{equation}\label{decompositionofOmega}
\MOmN = \left( \Xi_{-J} \Xi_{-J+1} \cdots \Xi_{J_1}  \right)  \left( \Xi_{J_1+1} \cdots \Xi_{J+1} \right) = \Omega^{(1)} \Omega^{(2)},
\end{equation}
where $\Omega^{(1)} =  \Xi_{-J} \Xi_{-J+1} \cdots \Xi_{J_1} $ and $\Omega^{(2)} = \Xi_{J_1+1} \cdots \Xi_{J+1}$. By Lemma \ref{lemma:controloneigenvalues}, for any $g_1 \in \Omega^{(1)}$, and $g_2 \in \Omega^{(2)}$, we have
\begin{equation}\label{inequalityofg1g2}
\frac{1}{16000} < \frac{\lambda(g_1)}{H_1} <32 \text{, and} \quad  \frac{1}{16000} < \frac{\lambda(g_2)}{H_2} < 32,
\end{equation}
where $H_1 = N^{3/4}$ and $H_2 = N^{1/4}$. Now we define the measure $\mu$ and $\nu$ on $\Z^2$ by
$$
\mu (x) : = \sum_{g_1 \in \Omega^{(1)}} \textbf{1}_{\{x = ^tg_1 \gamma_a e_2\}},
$$
$$
\nu (y) : = \sum_{g_2 \in \Omega^{(2)}} \textbf{1}_{\{y = g_2 e_2\}},
$$
with $\mu,\nu \leq 1$. Writing 
$$
S_{N,a}(\theta) = \sum_{x}\sum_{y} \mu(x) \nu(y) e(\theta \langle x,y \rangle),
$$
we proceed to bound 
\begin{align*}
\sum_{\theta \in Z} \norm{S_{N,a}(\theta)} 
&=\sum_{\theta \in Z} \zeta(\theta) S_{N,a}(\theta) \\
&= \sum_{\theta = \frac{s}{q} + \frac{\frl}{\calT N'} \in Z}  \zeta(\theta) \sum_{x}\sum_{y} \mu(x) \nu(y) e(\theta \langle x,y \rangle),
\end{align*}
which $\zeta$ has modulus 1.  Consider a non-negative bump function $\Upsilon$, see $\S \ref{testfunction} $, which is at least one on ${[-1,1]}^2$, and has Fourier transformation supported in a ball of radius $1/\left( 2^{28}(A+1)^2 \right)$ about the origin.  Now, we apply Cauchy-Schwarz in the sum $x$, insert the function $\Upsilon$, and use Poisson summation. Finally, by the fact that $\# \Omega^{(1)} \ll H_1^{2 \delta}$, see Hensley \cite{Hen:89}, we obtain the following inequality.
\begin{equation}\label{firststepinsecondrefinement}
\begin{split}
\sum_{\theta \in Z}\norm{S_{N,a}(\theta)} 
&\ll N^{3\delta / 4}   \ppowtwo{ \sum_x \Upsilon\left( \frac{x}{32(A+1)H_1} \right) \powtwo{ \norm{  \sum_{\theta = \frac{s}{q} + \frac{\frl}{\calT N'}\in Z}  \zeta(\theta)\sum_{y} \nu(y) e(\theta \langle x,y \rangle) }  }   } \\
&\ll N^{3(\delta+1)/4} \chi^{1/2},
\end{split}
\end{equation}
where
\begin{equation}\label{definitionofchi}
\chi  = \chi_{Q,K} :=  \sum_{\theta \in Z} \sum_{\theta ' \in Z}  \nu(y) \nu(y') \textbf{1}_{ \left\{ \dnorm{y\theta -y' \theta '   } < \frac{1}{2^{28}(A+1)^2H_1}\right\} } .
\end{equation}
Here $\theta '= \frac{s'}{q'}+\frac{\frl '}{\calT N'}$.
 Also, we write $y = (y_1,y_2)$, and the same with $y'$. Recall that $ \Gamma$ contains only hyperbolic elements; hence we have 
$$
y_1 y_1' y_2 y_2' \neq 0.
$$

On the other hand, besides the innermost condition in (\ref{definitionofchi}), we have
\begin{equation}\label{inequalityofbeta}
\norm{y_1 \frac{\frl}{\calT N'} - y_1' \frac{\frl '}{\calT N'} } \leq \norm{y_1 \frac{\frl}{\calT N'} } + \norm{y_1' \frac{\frl '}{\calT N'} } \leq \frac{64H_2 K}{N'},
\end{equation}
and the same with $\norm{y_2 \frac{\frl}{\calT N'} - y_2' \frac{\frl '}{\calT N'} }$. 
Thus, we have the following inequality:
\begin{equation}\label{secondstep}
\dnorm{y_1\frac{s}{q} - y_1' \frac{s ' }{q '} } \leq \dnorm{ y_1 \theta -y_1' \theta } + \norm{y_1 \frac{\frl}{\calT N'} - y_1' \frac{\frl '}{\calT N'} } \leq \frac{1}{2^{28}(A+1)H_1} + \frac{64H_2K}{N'},
\end{equation}
and similarly with $y_2,y_2'$.

Let $Y: = \lmx  y_1& y_1'\\y_2 & y_2' \rmx$, so that 
\begin{equation}\label{determinantofY}
\mathcal{Y} : = \text{det}(Y) = y_1y_2' - y_2 y_1'.
\end{equation}
Observe then by  (\ref{inequalityofg1g2}), (\ref{secondstep}), $Q<{\left( 2^{26}(A+1)^2  N  \right)  }^{1/2}$, and $KQ<{\left( 2^{26}  (A+1)^2N  \right)  }^{1/2}$ that 
\begin{equation*}
\begin{split}
\dnorm{\calY \frac{s}{q}} 
&\leq \dnorm{y_2' \left( y_1\frac{s}{q} - y_1' \frac{s ' }{q '}  \right) } + \dnorm{y_1' \left(  y_2'\frac{s'}{q'} - y_2 \frac{s  }{q }   \right)   } \\
& \leq 32(A+1)H_2 \left(   \frac{1}{2^{28}(A+1)^2H_1} + \frac{64H_2K}{N'}   \right)  \times 2 \\
& < \frac{1}{Q}.
\end{split}
\end{equation*}

This forces $\calY \equiv 0 \pmod{q}$. The same arguments gives $\calY \equiv 0 \pmod{q'}$, and hence we have 
\begin{equation}\label{thirdstep}
\calY \equiv 0 \pmod{\mathfrak{q}},
\end{equation}
where $\frac{1}{2}Q \leq \frq < Q^2$ is the least common multiple of $q$ and $q'$. 

Decompose $\chi$ in (\ref{definitionofchi}) as $\chi_1+\chi_2$ according to whether $\calY = 0$ or not; we handle these two contributions separately. 

\vspace{2mm}
\subsection{Bounding $\chi_1$: the case $\calY = 0$.}~

The condition $\calY = 0$ implies that $y_1/y_2 = y_1' / y_2'$. Recall that rationals have unique continued fraction expansions (of even length), and thus $y=y'$. Now, set
$$
\tilde{r} = \frac{1}{2^{28}(A+1)^2H_1} + \frac{64H_2K}{N'},
$$
and 
$$
\alpha = \norm{\frac{s}{q}-\frac{s'}{q'}} .
$$
The equation (\ref{secondstep}) now becomes
\begin{equation}\label{fourthstep}
\dnorm{ y \alpha  } \leq \tilde{r}.
\end{equation}

Furthermore, the above equation can be rewritten as 
\begin{equation}
y_1 \alpha = n_1 + t_1\tilde{r} \text{, and} \quad y_2 \alpha = n_2 + t_2 \tilde{r}, 
\end{equation}
where $\norm{t_1}, \norm{t_2} \leq 1 $. Fix $\theta = \frac{s}{q} + \frac{\frl}{\calT N'}$, we will show that $\theta ' $ has $\ll \calT$ choices. To prove this, we consider two cases - $\alpha = 0$, or $\alpha \neq 0$. 

 \textbf{1. $\alpha = 0$. } This implies that 
$$
\norm{y \left( \frac{\frl - \frl '}{\calT N'} \right)} = \norm{y\left( \theta -\theta '  \right)}  \leq \frac{1}{2^{28} H_1}.
$$
Consequently, we have $\ll \calT$ choices for $\frl '$.

\textbf{2. $\alpha \neq 0$.} Straightforward computation shows that 
\begin{equation}\label{fourthstep2}
(y_1n_2 - y_2 n_1) = (t_1y_2 - t_2 y_1)\tilde{r}.
\end{equation}
Hence, 
\begin{align*}
\norm{ y_1n_2 - y_2 n_1 } = (y_1+y_2)\tilde{r} \leq  32(A+1)H_2 \left(   \frac{1}{2^{28}(A+1)^2H_1} + \frac{64H_2K}{N'}   \right)  \times 2  < \frac{1}{Q} < 1
\end{align*}
Thus, $y_1n_2 - y_2 n_1 = 0$, and by the fact that $(y_1,y_2) = 1$, we obtain 
$n_1 = y_1 t$, and $n_2 = y_2t$. However, since $\alpha < 1$, the only choices for $t$ are 0 or 1. 
\begin{enumerate}
\item \textbf{t = 0.} We have
$$
y_2 \alpha =t_2\tilde{r},
$$
which implies 
$$
\alpha < \frac{16000}{H_2}\tilde{r}.
$$
\item \textbf{t = 1.} We have
$$
y_2 \alpha = y_2 + t_2\tilde{r} \Rightarrow y_2 (1-\alpha ) = -t_2 \tilde{r},
$$
which implies 
$$
1- \alpha <  \frac{16000}{H_2}\tilde{r}.
$$
\end{enumerate}
Since $K \geq \widetilde{C} > 2^{20}$, we have $ \frac{16000}{H_2} \tilde{r} < \frac{1}{Q^2}$, and thus both case (1) and case (2) have no solution with $\alpha \neq 0$. 

From the above discussion, we conclude that
\begin{equation}\label{chi1result}
\norm{\chi_1} \ll \norm{\Omega^{(2)}} |Z| \calT,
\end{equation}
where $\norm{\Omega^{(2)}}$ is the number of choices for $g_2$, $|Z|$ is the number of choices for $\theta$, and $\calT$ is the number of choices for $\theta ' $.

\begin{remark}\label{remark:7.1}
Minor modification is needed when $K$ is at constant level. The case $\alpha = 0$ directly implies that $\theta ' $ is fixed. The proof for the case $\alpha \neq 0$ remain unchanged. Then, we have 
$$
\norm{\chi_1} \ll \norm{\Omega^{(2)}} |Z|.
$$
\end{remark}

\vspace{2mm}

\subsection{Bounding $\chi_2$: the case $\calY \neq 0$. }~

Note that $\norm{\calY} \leq {(32H_2)}^2$. Since $\frq | \calY$, and $\calY \neq 0$, we have
$$
\frq \leq Q^2\text{, }\frq \leq {(32H_2)}^2,
$$
which means
$$
 \frac{1}{\frq} \geq \frac{1}{32H_2Q}.
$$
Moreover, from the definition of $H_1$ and $H_2$, one can easily show that 
$$
\frac{1}{2^{28}(A+1)^2H_1} < \frac{1}{64H_2Q}, \quad \frac{64H_2K}{N'} < \frac{1}{64H_2Q}.
$$
Hence, (\ref{secondstep}) implies that  
$$
\dnorm{y_1\frac{s}{q} - y_1' \frac{s ' }{q '} }  < \frac{1}{32H_2Q} \leq \frac{1}{\frq}.
$$
This forces 
\begin{equation}\label{forcecongruence}
y_1\frac{s}{q} - y_1' \frac{s ' }{q '} \equiv 0 \pmod{1}, \text{and} \quad  \norm{y_1 \frac{\frl}{\calT N'} - y_1' \frac{\frl '}{\calT N'} } \leq \frac{1}{2^{28}(A+1)^2H_1}
\end{equation}
and the same holds for $y_2,y_2'$. Let $\tilde{q} : = (q,q')$ and $q = q_1 \tilde{q}$, $q' = q'_1 \tilde{q}$ so that $\frq = q_1q'_1 \tilde{q}$. Then (\ref{forcecongruence}) becomes 
$$
y_1 s q'_1 \equiv y'_1 s' q_1 \pmod{\frq},
$$
and the same for $y_2,y'_2$. Recall that $s$ and $q$ are coprime, as are $s'$ and $q'$. It then follows that $q_1 | y_1,$ and similarly, $q_1 |y_2$. But since $y$ is a visual vector, we need $q_1 = 1$. The same arguement applies to $q'_1$, so we have $q = q' = \frq$. Then  (\ref{forcecongruence}) now reads 
\begin{equation}\label{newcongruence}
y_1s \equiv y'_1 s' \pmod{q},
\end{equation}
and similarly for $y_2,y'_2$.

We start by fixing $g'_2$ for which there are $|\Omega^{(2)}|$ choices. Hence, the vector $y' $ is fixed. Next, we fix $\theta \in Z$ for which there are $|Z|$ choices, and thus, the parameters $s$, $q$, and $\frl$ are fixed. Notice that from (\ref{forcecongruence}), we have
\begin{equation}
\frac{\calY \frl}{\calT} =   \frac{ \norm{y_1y_2'-y_2y_1'} \frl  }{\calT} \leq y_1' \norm{\frac{y_2 \frl - y_2' \frl '}{\calT}} + y_2' \norm{\frac{y_1' \frl ' - y_1 \frl}{\calT}} \ll \frac{N'}{2^{28}(A+1)^2H_1} \cdot 64 H_2 \ll H_2^2.
\end{equation}
This implies that 
$$
\norm{y_1y'_2 - y_2 y'_1} \ll \frac{H_2^2}{K}.
$$
Since $q | \calY$, there are $\ll \frac{H_2^2}{KQ}$ choices for $t'$ (note that $KQ \ll H_2^2$), where 
\begin{equation}\label{solutiontot'}
y_1 y'_2 - y_2 y'_1 = t'q. 
\end{equation}
Now, we fix $t'$. Again, with $y$ and $y'$ being visual vectors, all solutions to (\ref{solutiontot'}) are of the following form
\begin{equation}\label{generalsolution}
y_1 = y^{\ast}_1 + t y'_1\text{, } y_2 = y_2^{\ast} + t y'_2,
\end{equation}
where $t \in \Z$ and $(y^{\ast}_1,y_2^{\ast} )$ is a solution to (\ref{solutiontot'}). (Note that $y'$ is already fixed.)

In addition, we have 
$$
 \norm{\frac{y_2 \frl - y'_2 \frl '}{\calT}  } \leq \frac{N'}{2^{28}(A+1)^2H_1}.
$$
Dividing both side by $\frac{y_2'}{\calT}$, we get
$$
\norm{\frl ' - \frac{y_2}{y'_2} \frl} \leq \frac{\calT}{y'_2} \frac{N'}{2^{28}(A+1)^2H_1} \leq 8000 \calT.
$$
After we take away the absolute sign and use \eqref{generalsolution}, the above equation becomes
\begin{equation}\label{choicesforfel'}
\begin{split}
\frac{y^{\ast}_2}{y'_2} \frl + t\frl - 8000 \calT \leq \frl' \leq \frac{y^{\ast}_2}{y'_2} \frl + t\frl + 8000 \calT.
\end{split}
\end{equation}
Since $\frl , \frl \asymp K \calT$, there are only $\ll 1$ choices for $t$. Fix $t$, and so $y$ is fixed by \eqref{generalsolution}. 

Finally, with $t$ fixed, there are $\ll \calT$ choices for $\frl '$. 
Therefore, the above discussion shows
\begin{equation}\label{chi2result}
\norm{\chi_2 } \ll \norm{\Omega^{(2)}} |Z| \frac{H_2^2}{KQ} \calT. 
\end{equation}

\begin{remark}\label{remark:7.2}
When $K$ is at constant level, we stop at \eqref{newcongruence}. That is, for fixed $y'$, there are $\ll \frac{H_2^2}{Q}$ choices for $y$. Then, from \eqref{newcongruence}, $s'$ is uniquely determined by $s$. Consequently, we have
$$
\norm{\chi_2 } \ll \norm{\Omega^{(2)}} |Z| \frac{H_2^2}{Q} . 
$$
\end{remark}

\vspace{2mm}

\subsection{Combining $\chi_1$ and $\chi_2$}
First of all, from the previous subsections, the upper bound of $\norm{\chi_2}$ dominates the one of $\norm{\chi_1}$ since $KQ \ll H_2^2$. Therefore, we have 
$$
\norm{\chi } = \norm{\chi_1+\chi_2} \ll \norm{\Omega^{(2)}} |Z| \frac{H_2^2}{KQ} \calT.
$$
In addition, (\ref{firststepinsecondrefinement}) becomes 
\begin{equation}\label{lastbound2}
\sum_{\theta \in Z} \norm{S_{N,a}(\theta)} \ll 
\frac{N^{1+3\delta/4} { \norm{\Omega^{(2)}} }^{1/2} }{\sqrt{KQ}} {|Z|}^{\frac{1}{2}}  {\calT}^{\frac{1}{2}}.
\end{equation}
Finally, using Lemma \ref{lemma:intro1},  we prove (\ref{firstminorarc}).
\end{proof}

Once we have the universal bound for any subset $Z \subset P_{Q,K}$,  Lemma \ref{trickCZ} implies the following theorem.
\begin{theorem}\label{Tkqbig}
Assume that $Q < {\left( 2^{26}(A+1)^2N \right) }^{1/2}$ and $KQ <{\left( 2^{26}(A+1)^2N \right) }^{1/2}$. Then for any $\epsilon > 0$, 
\begin{equation}\label{eq:prefinalminorarc1}
\frac{1}{\calT N'} \sum_{\theta \in P_{Q,K}} \powtwo{\norm{S_{N,a}(\theta)}} \ll_{\epsilon} \frac{\ppowtwo{\# \MOmN} }{N} \frac{N^{2(1-\delta) + \epcon + \epsilon }}{KQ}.
\end{equation}
That is to say, we have
\begin{equation}\label{eq:finalminararc1}
\int_{W_{Q,K}} \norm{S_{N,a}(\theta)}^2 d \theta  \ll_{\epsilon} \frac{\ppowtwo{\# \MOmN} }{N} \frac{N^{2(1-\delta) + \epcon + \epsilon }}{KQ}.
\end{equation}
\end{theorem}
\begin{remark}
It is not hard to show that \eqref{eq:finalminararc1} still holds for the case $ \widetilde{C} \leq K < 2 \widetilde{C}$.
\end{remark}

\section{Error Term Analysis: Small $KQ$}
In this section, we estimate the error term when $KQ$ is small. Instead of using a ``triple'' Kloosterman refinement, we use only a double refinement. Specifically, we leave $a$ fixed, and estimate the sum over $q$ and $\frl$. Again, we only give detailed discussion on the case $K \geq 2 \widetilde{C}$.

For each region $W_{Q,K}$, define a set $P_{Q,K,s}$ as follows.
\begin{equation}\label{setqka}
P_{Q,K,s} = \left\{ \frac{s}{q}+\frac{\frl}{\calT N'}: \frac{1}{2}Q \leq q < Q\text{, }q \geq s\text{, } \frac{1}{2}K \calT \leq \frl < K \calT    \right\}.
\end{equation}
Hence the size of $P_{Q,K,s}$ is $KQ\calT$. We now bound $\sum_{\theta \in P_{Q,K}} \powtwo{\norm{S_{N,a}(\theta)}}$ by $Q$ times $\sum_{\theta \in P_{Q,K,s}} \powtwo{\norm{S_{N,a}(\theta)}}$. Again, we are looking for a universal bound over $\sum_{\theta \in Z} \norm{S_{N,a}(\theta)}$ for any subset $Z \subset P_{Q,K,s}$.

\begin{theorem}\label{Tunionboundkqsmall}
Assume that 
\begin{equation}\label{conditionofkq}
KQ \leq N^{2(1-\delta)+2\epcon}.
\end{equation}
Then for any subset $Z \subset P_{Q,K,s}$ and any $\epsilon > 0$, we have
\begin{equation}\label{unionboundkqsmall}
\sum_{\theta \in Z} \norm{S_{N,a}(\theta)} \ll \# \MOmN \cdot \frac{ {(K^{3/2}Q^3)}^{1-\delta} {(K^{3/2}Q^3)}^{4 \epcon + \epsilon}  }{\sqrt{KQ^2}} \cdot {|Z|}^{\frac{1}{2}} \cdot \calT^{\frac{1}{2}}
\end{equation}
\end{theorem}

\begin{remark}
When $K$ is at constant level, we have instead
$$
\sum_{\theta \in Z} \norm{S_{N,a}(\theta)} \ll \# \MOmN \cdot \frac{ {(KQ^3)}^{1-\delta} {(K^{3/2}Q^3)}^{4 \epcon + \epsilon}  }{Q} \cdot {|Z|}^{\frac{1}{2}}.
$$
\end{remark}

\begin{proof}
Our goal is to decompose $\MOmN$ as follows.
\begin{equation}\label{decomofomega}
\MOmN  = \Omega^{(1)} \Omega^{(2)} = \Omega^{(1)} \left[ \Omega^{(3)}\Omega^{(4)} \Omega^{(5)}  \right],
\end{equation}
where 
\begin{equation*}
\begin{split}
& \Omega^{(1)} = \Xi_{-J} \Xi_{-J+1} \cdots \Xi_{j_1}, \\
&  \Omega^{(2)} = \Xi_{j_1+1} \Xi_{j_1+2} \cdots \Xi_{J+1}, \\
&  \Omega^{(3)} =\Xi_{j_1+1} \Xi_{j_1+2} \cdots \Xi_{j_2}, \\
&  \Omega^{(4)} =\Xi_{j_2+1} \Xi_{j_2+2} \cdots \Xi_{h-1}, \\
&  \Omega^{(5)} =\Xi_{h} \Xi_{h+1} \cdots \Xi_{J+1}.\\
\end{split}
\end{equation*}
The parameters $j_1$, $j_2$ and $h$ will be determined later.     

First of all, it is easy to verify that 
$$
{ \left( KQ^2 {\left(KQ^2\right)}^{2 \epcon} \sqrt{K}  \right) }^{1+2\epcon} < K^{3/2}Q^2{ (K^6Q^{12}) }^{\epcon}.
$$
Thus, by \eqref{eq:epcondef} and  (\ref{conditionofkq}), we have
\begin{equation}\label{kqinitialcondition}
{ \left( KQ^2 {\left(KQ^2\right)}^{2 \epcon} \sqrt{K} \right) }^{1+2\epcon}  < N^{2/3} 
\end{equation}
which implies
\begin{equation}\label{kqinitialcondition22}
  KQ^2 {\left(KQ^2\right)}^{2 \epcon} \sqrt{K}  < N^{1-\epcon}.
\end{equation}

Hence, by Corollary \ref{cor:mfallinsize}, there exist parameters $-J+1 \leq j_1, j_2 ,h \leq J-1$ such that 
\begin{equation}\label{h1h3h5}
\begin{split}
KQ^2 \leq & \calN_{j_1} \leq {\left(  KQ^2 \right)}^{1+2\epcon}, \\
KQ^2 { \left( KQ^2 \right)}^{2\epcon} \sqrt{K} \leq & \calN_{j_2} \leq {\left( KQ^2 { \left( KQ^2 \right)}^{2\epcon} \sqrt{K}  \right)}^{1+2\epcon} ,\\
Q \leq &\frac{N}{\calN_{h-1}} \leq {Q }^{1+ 2\epcon}.
\end{split} 
\end{equation}
We set 
\begin{equation}
\begin{split}
H_1 = \calN_{j_1}, \quad H_2 = N/\calN_{j_1}, \quad H_3 = \calN_{j_2}/ \calN_{j_1}, \\
H_4 = \calN_{h-1}/ \calN_{j_2}, \quad \text{and} \quad H_5 = N/\calN_{h-1}.
\end{split}
\end{equation}
\noindent
It is then elementary to show that
\begin{equation}\label{eq:H3K}
H_3 = \frac{\calN_{j_2}}{\calN_{j_1}} \geq  \sqrt{K}.
\end{equation}
\noindent
Now, to show that the decomposition is legit, we need 
\begin{enumerate}
\item $h \geq j_2+2$. We prove $j_2 < \bbj < h$ instead. The first inequality follows from \eqref{eq:specialindex} and \eqref{kqinitialcondition}  which
\begin{equation*}
\calN_{\bbj} \geq N^{2/3} > {\left( KQ^2 { \left( KQ^2 \right)}^{2\epcon} \sqrt{K}  \right)}^{1+2\epcon}  \geq \calN_{j_2}
\end{equation*}
The second inequality holds since we have
\begin{equation*}
\frac{N}{\calN_{\bbj}} \geq N^{\frac{1}{3}- \frac{4\epcon}{3}} > Q^{1+2\epcon} \geq \frac{N}{\calN_{h-1}}.
\end{equation*}
\item $j_1 < j_2$. One see from \eqref{h1h3h5} that the index $j_1$ satisfies 
$$N_{j_1-1} \leq KQ^2 \leq N_{j_1}.$$
Hence we have, 
$$
KQ^2 { \left( KQ^2 \right)}^{2\epcon} \sqrt{K} > {(KQ^2)}^{1+2\epcon} \geq N_{j_1-1}^{1+2\epcon} > N_{j_1},
$$
where the last inequality comes from the fact that $N_{j_1-1} \geq N_{j_1}^{1-\epcon}$. The above inequalities then imply  $j_1 < j_2$.
\end{enumerate}
Finally, since $j_2 < \bbj < h$, the special set $\aleph$ belongs to $\Omega^{(4)}$. This fact plays a crucial rule for the error term analysis.

We now follow the similar argument in Theorem \ref{Tfirstminorarc} and obtain 
\begin{equation}
\sum_{\theta \in Z} \norm{S_{N,a}(\theta)} \ll \sqrt{\# \Omega^{(1)}} H_1 \sqrt{\chi},
\end{equation}
where 
\begin{equation}\label{newchicondition}
\chi = \#  \left\{  g_2,g'_2 \in \Omega^{(2)}, \theta,\theta ' \in Z: \dnorm{\left( g_2\theta - g'_2 \theta ' \right) e_2} \leq \frac{1}{2^{28} H_1}   \right\}.
\end{equation}
Besides the innermost conidtion in (\ref{newchicondition}), we also have
\begin{equation}\label{newfrlcondition}
\norm{ \left( g_2 \frac{\frl}{\calT N'} - g'_2 \frac{\frl ' }{\calT N'}  \right) e_2 } \leq 2\cdot \frac{32H_2 K \calT}{\calT N' } \leq \frac{K}{10H_1} < \frac{1}{2Q^2}.
\end{equation}
Consequently,
\begin{equation}
\begin{split}
\dnorm{ \left( g_2 \frac{s}{q} - g'_2 \frac{s}{q'} \right) e_2  }  \leq \dnorm{\left( g_2\theta - g'_2\theta ' \right) e_2} + \norm{ \left( g_2 \frac{\frl}{\calT N'} - g'_2 \frac{\frl ' }{\calT N'}  \right) e_2 }  < \frac{1}{Q^2}.
\end{split} 
\end{equation}
This implies 
$$
\left( g_2 \frac{s}{q} - g'_2 \frac{s}{q'} \right) e_2 \equiv 0 \pmod{1}
$$
Again, we can conclude that 
\begin{equation}\label{newqdefinition}
q = q', \quad g_2e_2 \equiv g'_2 e_2 \pmod{q},
\end{equation}
and 
\begin{equation}\label{newbetadiff}
\norm{\frac{g_2\frl e_2 - g'_2 \frl 'e_2}{\calT N'}} = \dnorm{\left( g_2\theta - g'_2\theta ' \right) e_2} \leq \frac{1}{2^{28}H_1}.
\end{equation}
We start by fixing $g_2'$ for which there are $\# \Omega^{(2)}$ choices. Denote $g'_2 e_2 $ by $v'_2 = (x'_1,x'_2)$ and $g_2e_2$ by $v_2 = (x_1,x_2)$. Notice that $v'_2$ is now fixed. Next, we fix $\theta \in Z$ for which there are $|Z|$ choices. Hence, (\ref{newqdefinition}) and (\ref{newbetadiff}) now read
\begin{equation}
v_2 \equiv v_2' \pmod{q}, \quad \norm{\frac{x_1\frl - x'_1 \frl '}{\calT N'}},\norm{\frac{x_2\frl - x'_2 \frl '}{\calT N'}} \leq \frac{1}{2^{28}H_1}.
\end{equation}
We write $g_2 = g_3 g_4g_5$, where $g_3 \in \Omega^{(3)}$, $g_4 \in \Omega^{(4)}$, and $g_5\in \Omega^{(5)}$. Also, let $g_3$ be $\lmx t_1 & t_2 \\ t_3& t_4 \rmx $, and $g_4g_5 e_2 $ be $(y_1,y_2)$. Then we naturally have
$$
x_1 = t_1y_1+t_2y_2, \quad x_2 = t_3y_1 + t_4y_2.
$$
This implies that $x_1/x_2$ and $b/d$ are close:
\begin{equation}\label{diffoffraction}
\frac{x_1}{x_2} = \frac{t_1y_1+t_2y_2}{t_3y_1 + t_4y_2} = \frac{t_2}{t_4} + \frac{y_1}{(t_3y_1+t_4y_2)t_4}, \quad \text{and} \quad \norm{\frac{x_1}{x_2} - \frac{t_2}{t_4}} \leq \frac{1}{t_4^2}.
\end{equation}
On the other hand, we have
\begin{equation}\label{diffoffraction2}
\begin{split}
\norm{\frac{x_1}{x_2} -\frac{x'_1}{x'_2} }  = &\norm{  \frac{x_1 \frl}{x_2 \frl} -\frac{x'_1 \frl'}{x'_2\frl'} }  \\
& =  \norm{\frac{x'_2\frl '(x_1\frl - x'_1\frl ')+x'_1\frl '(x'_2\frl '-x_2\frl ) }{x_2x_2'\frl \frl '} }  \\
& \leq \frac{\calT N}{2^{28}H_1} \cdot \frac{4}{H_2K \calT} \ll \frac{1}{K}.
\end{split}
\end{equation}
Combining (\ref{diffoffraction}) and (\ref{diffoffraction2}), we get 
\begin{equation}\label{diffoffraction3}
\norm{\frac{t_2}{t_4} - \frac{x'_1}{x'_2}} \ll \frac{1}{H_3^2} + \frac{1}{K}
\end{equation}

Finally, by the fact that the fraction $\frac{t_2}{t_4}$ uniquely determines $g_3$ and each two distinct fractions $\frac{t_2}{t_4}$ and $\frac{t_2'}{t_4'}$ have difference at least $\frac{1}{t_4t_4'} \gg \frac{1}{H_3^2}$, we conclude that there are $\ll \frac{H_3^2}{K}$ choices for $\frac{t_2}{t_4}$, and thus for $g_3$. (Here we use the fact that $H_3 \geq \sqrt{K}$.)

Now, we fix the element $g_3$, and fix the next element $g_4$ for which there are $\# \Omega^{(4)}$ choices. We denote $g_3g_4$ as $\lmx m_1 & m_2 \\ m_3 & m_4 \rmx$, and $g_5 e_2 = z = (z_1,z_2)$. Since $g_3g_4 \in \text{SL}(2,\Z)$,  the vector $z$ satisfies the following equations.
$$
z_1 \equiv m_4x'_1 - m_2 x'_2 \pmod{q}, \quad z_2 \equiv -m_3x'_1 + m_1 x'_2 \pmod{q}.
$$
Again we have $H_5 \geq Q$ and $\lambda(g_5) \asymp H_5$, and thus, there are $\ll \frac{H_5^2}{Q^2}$ choices for $g_5$. Finally, from (\ref{newbetadiff}), there are $\ll \calT$ choices for $\frl '$.

Combining all the estimates for $g_2,g_2', \theta, \theta'$, we get 
\begin{equation}\label{finalbound1refine}
\begin{split}
\sum_{\theta \in Z} \norm{S_{N,a}(\theta)} 
& \ll \# \MOmN \cdot \frac{ (H_1H_3H_5)^{1-\delta} 2^{c\ppowtwo{\log \log (H_1H_3H_5)} } }{\sqrt{KQ^2}} \cdot \calT^{\frac{1}{2}} \cdot {|Z|}^{\frac{1}{2}}, \\
& \ll \# \MOmN \cdot \frac{ {(K^{3/2}Q^3)}^{1-\delta} {(K^{3/2}Q^3)}^{4 \epcon + \epsilon}  }{\sqrt{KQ^2}} \cdot {|Z|}^{\frac{1}{2}} \cdot \calT^{\frac{1}{2}}
\end{split}
\end{equation}
\end{proof}

\begin{remark}
Even when $K$ is at constant level, we still have \eqref{newqdefinition}. After fixing $g_2'$, we then directly fix $g_3$ and $g_4$ so that we do not gain a $K$ saving from $g_3$. The rest argument stays the same. Hence, \eqref{finalbound1refine} becomes
$$
\sum_{\theta \in Z} \norm{S_{N,a}(\theta)} 
 \ll \# \MOmN \cdot \frac{ (H_1H_5)^{1-\delta} 2^{c\ppowtwo{\log \log (H_1H_5)} } }{Q} \cdot {|Z|}^{\frac{1}{2}}.
$$
\end{remark}

We now have a general bound for $\sum_{\theta \in Z} \norm{S_{N,a}(\theta)}$, and Lemma \ref{trickCZ} implies the following theorem
\begin{theorem}\label{Tkqsmall}
Assume that 
$$
KQ \leq N^{2(1-\delta)+2\epcon}.
$$
Then for any $\epsilon >0,$
\begin{equation}\label{eq:prefinalminorarc2}
\frac{1}{\calT N'} \sum_{P_{Q,K}} \powtwo{\norm{S_{N,a}(\theta)}} \ll_{\epsilon} \frac{\ppowtwo{\# \Omega_N}}{N}\frac{ {\left(  K^{3/2} Q^3 \right)}^{2(1-\delta) }{(KQ)}^{\epsilon}   }{KQ}
\end{equation}
Hence, we have
\begin{equation}\label{eq:finalminorarc2}
\int_{W_{Q,K}} \norm{S_{N,a}(\theta)}^2 d\theta \ll_{\epsilon} \frac{\ppowtwo{\# \Omega_N}}{N}\frac{ {\left(  K^{3/2} Q^3 \right)}^{2(1-\delta) }{(KQ)}^{\epsilon}   }{KQ}.
\end{equation}
\end{theorem}

\begin{remark}
Again, \eqref{eq:finalminorarc2} remains the same when $K$ is at constant level.
\end{remark}

\section{Error Term Analysis: Proof of Theorem \ref{thm:minorterm}}
First of all, we combine the results in previous two sections , and obtain
\begin{theorem}\label{thm:erroranalysis}
For each region $W_{Q,K}$ defined in either \eqref{eq:Wqk2} or \eqref{eq:Wqk10}, we have
\begin{equation}\label{eq:finalminorarcbound}
\int_{W_{Q,K}} \norm{S_{N,a}(\theta)}^2 d\theta \ll \frac{\ppowtwo{\# \Omega_N}}{N}\frac{ 1}{K^{c_1}Q^{c_2}},
\end{equation}
for some sufficiently small absolute constants $0<c_1 < 1- 3(1-\delta)$ and $0<c_2 < 1-6(1-\delta)$.\end{theorem}
\begin{proof}
Write $Q = N^{\alpha}$, $K = N^{\kappa}$, with the parameters $(\alpha,\kappa)$ ranging in 
\begin{equation}\label{eq:rangeofalpha}
0 \leq \alpha < 1/2 \text{ and } 0 \leq \kappa < 1/2 - \alpha. 
\end{equation}
We break the summation, $\sum_{ \substack{1 \leq Q < {N'}^{1/2} \\ \text{dyadic}}} \sum_{\substack{\widetilde{C} \leq K<\frac{{N'}^{1/2}}{Q} \\ \text{dyadic}  } }$, into the following two ranges:
\begin{equation*}
\begin{split}
&\calR_1 := \{(\alpha,\kappa): \alpha+\kappa > 2(1-\delta) + 2\epcon\}, \\
&\calR_2 := \left\{(\alpha,\kappa):  \alpha+\kappa \leq 2(1-\delta) + 2\epcon \right\}.
\end{split}
\end{equation*}
Clearly,  $\calR_1$ and $\calR_2$ cover the whole region \eqref{eq:rangeofalpha}. Moreover, It is easy to see from \eqref{eq:finalminararc1} and \eqref{eq:finalminorarc2} that \eqref{eq:finalminorarcbound} hold  in both ranges $\calR_1$ and $\calR_2$.
\end{proof}

Now, we give the

\begin{myproof}[Proof of Theorem \ref{thm:minorterm}]

By Parseval, we have
\begin{equation}
\sum_{n \in \Z} \norm{\calE_{N,a}(n)}^2 = \int_0 ^1 \norm{1 - \Psi_{\calQ,N}(\theta)}^2 \norm{S_{N,a}(\theta)}^2 d \theta = \int_{\frM_{\calQ} } + \int_{\frm},
\end{equation}
where we broke the integral into the major arcs $\frM_{\calQ}$ and the complementary minor arcs $\frm = [0,1] \backslash \frM_{\calQ} $. We also set 
$$
\calI_{Q,K} = \int_{W_{Q,K}} \norm{S_{N,a}(\theta)}^2 d \theta.
$$
To evaluate  $\int_{\frM_{\calQ} }$, \eqref{eq:triangle} implies that for $x \in [-1,1]$,
$$1- \psi(x) = \norm{x}. $$
Hence, using Theorem \ref{thm:erroranalysis}, we get
\begin{equation}\label{eq:minorboundl2}
\begin{split}
\int_{\frM_{\calQ}} 
& \ll \sum_{q< \calQ} \sum_{(a,q) =1} \int_{\norm{\beta} < \calQ/N } \norm{\frac{N}{\calQ}\beta}^2 \norm{S_{N,a}(\theta)}^2 d \theta \\
& \ll \sum_{\substack{Q < \calQ \\ \text{dyadic} }} \sum_{ i =0}^{\asymp \log \calQ}2^{-2i}  \calI_{Q,K} \\
& \ll  \sum_{\substack{Q < \calQ \\ \text{dyadic} }} \sum_{ i =0}^{\asymp \log \calQ}2^{-2i}  \frac{\ppowtwo{\# \Omega_N}}{N}\frac{ 1  }{K^{c_1}Q^{c_2}} \\
& \ll  \sum_{\substack{Q < \calQ \\ \text{dyadic} }} \frac{ 1  }{Q^{c_2}} \sum_{ i =0}^{\asymp \log \calQ}2^{-2i}  \frac{\ppowtwo{\# \Omega_N}}{N}\frac{ 1  }{K^{c_1}}  \ll \frac{\ppowtwo{\# \Omega_N}}{N}\frac{ 1  }{\calQ^c}.
\end{split}
\end{equation}
where for each $i$, we have $K = \calQ/2^i$. In particular, since $0<c_1 <1$, the term $ \sum_{ i =0}^{\asymp \log \calQ}2^{-2i}  \frac{\ppowtwo{\# \Omega_N}}{N}\frac{ 1  }{K^{c_1}} $ contributes at most $ \frac{{\#\Omega_N}^2}{N\calQ^{c_1}}$.

On the other hand, we decompose the integral of error function over the minor arcs $\frm$ into dyadic regions 
$$
\int_{\frm} \ll \sum_{ \substack{Q < N^{1/2} \\ \text{dyadic}}} \sum_{\substack{ K<\frac{N^{1/2}}{Q} \\ \text{dyadic}  } } \calI_{Q,K},
$$
where at least one of $Q$ or $K $ exceeds $\calQ$. 

Thus, dyadically summing over $Q$ and $K$, we obtain
\begin{equation}\label{eq:finalminorarcbound2}
\int_{\frm} \ll  \frac{\ppowtwo{\# \Omega_N}}{N}\frac{ 1  }{\calQ^c}.
\end{equation}
Finally, combining the above result with \eqref{eq:minorboundl2} completes the proof.\end{myproof}

\vspace{2mm}

\appendix
\section{Proof of Corollary \ref{Cor:goodapp}}\label{AppendixA}
We follow the proof in \cite{BK:2011}. First of all, the following Proposition is essential but easy to show. 

\begin{proposition}\label{prop:appendix}
Suppose that a fraction $b/d \in (0,1)$ has the following fraction expansion
$$
\frac{b}{d} = [a_1,a_2,\ldots, a_k]
$$
Then, we have
$$
\frac{d-b}{d} = \begin{dcases*}
       [1+a_2, a_3\ldots, a_k],  & when $a_1 = 1$.\\
     [1,a_1-1, a_2\ldots, a_k],    &  when $a_1 > 1$. \\
        \end{dcases*}
$$
\end{proposition}

Now, fix the alphabet $\calA = \{ 1,2,3,4,5,6\}$ which has Hausdorff dimension $\delta_{\calA} \approx 0.8676$, see \cite{Jen:04}. Let $S$ be the set of primes $p$ up to $N$ such that $p \equiv 3 \pmod{4}$, and every prime divisor of $(p-1)/2$ is larger than $N^{3/11}.$ Then, Theorem 25.11 in \cite{GBV} shows that
$$
\norm{S} \gg \frac{N}{{(\log N)}^2},
$$
where the implied constant is absolute.

Now, since the error term $e^{-c\sqrt{\log N}} $ in \eqref{eq:bourgaineffec2} is of size $o(1/{(\log N)}^2)$, the number of primes in $S$ represented by $\calA$ is large. That is, 
$$
\norm{ S \cap \tilde{\dd}_{\calA} } \gg \frac{N}{{(\log N)}^2}.
$$
We aim to show that for every prime $p \in S \cap \tilde{\dd}_{\calA}$, there exists a primitive root $b$ of $p$ such that $b/p $ is a Diophantine of height $7$.

Fix a prime $p \in S \cap \tilde{\dd}_{\calA}$ and a primitive root $c_p$ of $p$. Theorem \ref{thm:shinnyih} implies that $p$ appears with large multiplicity in $ \rr_{\calA}$. In particular, 
for any fixed constant $\frr > 0$ small enough, we have
\begin{equation}\label{eq:appen1}
\norm{ \{  b_i : b_i/p \in  \rr_{\calA}  \}} \gg N^{2\delta_{\calA} -1 - \frr},
\end{equation}
where the implied constant depends on $\frr$. 

To determine whether $b_i$ is a primitive root, we use the following observation. For each $1\leq  b_i \leq p-1$ with $b_i/p \in \rr_{\calA}$, there exists a unique integer $\mathfrak{b}_i$ such that $1\leq \mathfrak{b}_i \leq p-1$ and 
$$
b_i \equiv c_p^{\mathfrak{b}_i} \pmod{p}.
$$
Thus, $b_i$ is a primitive root of $p$ if and only if $(\mathfrak{b}_i ,p-1) =1$. Our goal is to exclude those $b_i$ with $(b_i,p) \geq 2$. 

First, count the number of $b_i$ with $(\mathfrak{b}_i,p-1) >2$. By the definition of $S$, some prime factor of $(p-1)/2$ must divides the exponent $\mathfrak{b}_i$. Since each prime factor of $(p-1)/2$ is of size $>N^{3/11}$, we get
\begin{equation}\label{eq:appen2}
\norm{ \{  b_i : b_i/p \in  \rr_{\calA} \text{, } (\mathfrak{b}_i , p-1) > 2 \} } \ll N^{8/11},
\end{equation}
where the implied constant is absolute.  Notice that with $\delta_{\calA} \approx 0.8676$, we have for $\frr > 0$ small enough,
\begin{equation}\label{eq:999}
2\delta_{\calA} -1 - \frr > 8/11.
\end{equation}

Now, \eqref{eq:appen1}, \eqref{eq:appen2}, and \eqref{eq:999} show that there must exists some $\tilde{b}$ such that $\tilde{b}/p \in  \rr_{\calA}$ and the corresponding $\tilde{\mathfrak{b}}$ satisfies $(\tilde{\mathfrak{b}} , p-1) = 1 \text{ or } 2$. We examine each case as follows.
\begin{enumerate}
\item $(\tilde{\mathfrak{b}} , p-1) = 1$. Then $\tilde{b}$ is a primitive root, and hence, we are done. 
\item $(\tilde{\mathfrak{b}} , p-1) = 2$. Since $p \equiv 3 \pmod{4}$, $p - \tilde{b}$ is a primitive root of $p$. Moreover, Proposition \ref{prop:appendix} shows that $(p - \tilde{b})/p \in \tilde{\dd}_{\{ 1,2,3,4,5,6,7 \}}$. 
\end{enumerate}
Combining the two cases, we prove Corollary \ref{Cor:goodapp}.

\begin{remark}
Due to current limit of sieving in almost prime, we are unable to use the alphabet $\calA = \{ 1,2,3,4,5\}$, see \eqref{eq:999}. In particular, Corollary \ref{Cor:goodapp} would be true for $A = 6$ if Chen's theorem \cite{Chen:1973} can be improved to 2-almost prime with prime factor of size $> N^{1/3}$.
\end{remark}

\section{Local Obstruction}\label{append:B}
 
In this section, we investigate the neccessary conditions for a finite alphabet $\calA$ to have $\frA_{\calA}  \neq \Z$. That is, for $\calA$ to have finite local obstructions. 
Specifically, we give an upper bound for the Hausdorff dimension $\delta_{\calA}$ of such alphabet $\calA$.  Recall from $\S 1.2$ that an integer $d$ is admissible, i.e. $d \in \frA_{\calA}$, if and only if 
\begin{equation}\label{eq:1000}
d \in  \dd_{\calA} =  \langle \calG_{\calA} \cdot e_2,e_2 \rangle \text{, for } \forall q \in \N.
\end{equation}

Hence, we say that $\calA$ has no \textbf{local obstructions} if $\frA_{\calA} = \Z$. 
Also, let $\Gamma_{\calA} \subset \text{SL}_2(\Z)$ be the determinant-one subsemigroup of $\calG_{\calA}$ which is freely and finitely generated by the matrix products 
\begin{equation}
\lmx 0 & 1 \\ 1 & a \\  \rmx \cdot \lmx 0 & 1 \\ 1 & a' \\ \rmx \text{, for } a, a' \in \calA.
\end{equation}
Again, we can ask whether the subsemigroup $\Gamma_{\calA}$ has \textbf{everywhere strong approximation} by which we mean 
$$\Gamma_{\calA}\equiv \text{SL}_2(q) \pmod{q}  \text{, for } \forall q \in \N.$$ 
Clearly, $\Gamma_{\calA}$ having everywhere strong approximation implies the alphabet $\calA$ having no local obstructions (but the converse need not hold). 

Our goal is to show the following

\begin{theorem}\label{thm:10000}
If $\Gamma_{\calA}$ does not have everywhere strong approxmination, then there exists some integer $k^{\ast} \geq 2$ and a residue $r \pmod{k^{\ast}}$ such that $\calA \subseteq r + k^{\ast}\Z $.
\end{theorem}

To prove Theorem \ref{thm:10000}, we need the following lemmas.  The first lemma shows that when $\calA$ contains $1$ and $2$, the semigroup $\Gamma_{\calA}$ has everywhere strong approximation, and hence no local obstructions. 

\begin{lemma}\label{lemma:1001}
If $\calA$ contains $1$ and $2$, then the semigroup $\Gamma_{\calA}$ has everywhere strong approximation. 
\end{lemma}

\begin{proof}
It is enough to show that the group generated by the matrices $a_1= \lmx 0 & 1 \\ 1 & 1 \\  \rmx$ and $a_2 = \lmx 0 & 1 \\ 1 & 2 \\  \rmx$ contains $\text{SL}_2(\Z)$. This is because after reduction mod $q$, $\Gamma_{\calA}$ becomes a group. This lemma now follows from the following equations.
$$
a_{1}^{-1}a_{2} = \lmx 1 & 1 \\ 0 & 1 \\  \rmx, \quad \text{ and } \quad a_{1}^{-1}a_{2}a_{1}^{-2}a_{2}a_{1}^{-1}a_{2}a_{1}^{-1} = \lmx 1 & 0 \\ 1 & 1 \\  \rmx.
$$
\end{proof}

In fact, Lemma \ref{lemma:1001} can be generalized as follows.

\begin{lemma}\label{lemma:1002}
If $\calA$ contains two consecutive positive integers $m$ and $m+1$, then the semigroup $\Gamma_{\calA}$ has everywhere strong approximation. 
\end{lemma}

\begin{proof}
Again, let $a_m$ be $\lmx 0 & 1 \\ 1 & m \\  \rmx$ and $a_{m+1}$ be $\lmx 0 & 1 \\ 1 & m+1 \\  \rmx$. We have 
$$
a_m^{-1}a_{m+1} = \lmx 1 & 1 \\ 0 & 1 \\  \rmx, \quad {(a_m^{-1}a_{m+1} )}^{m-1}a_m^{-1} = a_{1}^{-1}, \quad \text{ and } \quad {(a_m^{-1}a_{m+1} )}^{m-2}a_m^{-1} = a_{2}^{-1}.
$$
Consequently, the group generated by $a_m$ and $a_{m+1}$ contains $\text{SL}_2(\Z)$, and thus proves the claim.
\end{proof}

Now, we give the

\begin{myproof}[Proof of Theorem \ref{thm:10000}]

We can generalize the proof in Lemma \ref{lemma:1002} to show that if there exist integers $m_1,m_2,m_3,m_4\in \calA$ with $|m_1-m_2|$ and $|m_3-m_4|$ coprime, then $\Gamma_{\calA}$ has everywhere strong approximation. In fact, let $k = |m_1-m_2|$, for some $m_1,m_2 \in \calA$. Then, for any $m \in \calA$ and for all $q \in \N$, we have $\Gamma_{\{ m + k \Z \}} \subseteq \Gamma_{\calA} \pmod{q}$. Hence, let $k^{\ast}$ be the greatest common divisor of all possible difference $|m_1 - m_2|$ with $m_1,m_2 \in \calA$. Then, for any $m \in \calA$ and for all $q \in \N$, we have $\Gamma_{\{ m + k^{\ast} \Z \}} \equiv \Gamma_{\calA}  \pmod{q}$.

This implies that for $\Gamma_{\calA}$ to not have everywhere strong approximation, we need $k^{\ast} \geq 2$. That is, $\calA \subseteq r + k^{\ast} \Z$, for some $k^{\ast} \geq 2$ and residue $0 < r \leq k^{\ast} $. 
\end{myproof}

\begin{remark}
From the proof in Theorem \ref{thm:10000}, we see that when studying admissibility, the alphabet $\calA$ can be replaced by some arithmetic progression $m + k^{\ast} \Z$. In particular, $k^{\ast}$ is the greatest common divisor of all possible difference $|m_1 - m_2|$ with $m_1,m_2 \in \calA$.
\end{remark}

The next Lemma gives control over Hausdorff dimension on augmenting the alphabet. 

\begin{lemma}\label{lemma:1003}
Let $\calA$ be a finite subset of $\N$, and $n_1,n_2 \in \N$ be two integers not in $\calA$. Suppose that $n_1 < n_2$. We then have $\delta_{\calA \cup \{ n_1\}} > \delta_{\calA \cup \{ n_2\}}$. 
\end{lemma}
 
 \begin{proof}
 To compute the Hausdorff dimension $\delta_{\calA \cup \{ n_1\}}$, it is enough to consider the canonical open covers of $\frC_{\calA \cup \{ n_1\}}$ which are of the following form (see \cite{Hensley1996}). 
 $$
 \calC_k = \cup_{v \in V_{\calA \cup \{n_1\}}(k)} I(v),
 $$
 where for integers $k \geq 1$, $V_{\calA \cup \{n_1\}} $ is the set of $k-$tuple of integers with each entry belonging to $\calA \cup \{ n_1\}$, and for $v = (v_1,v_2,\ldots, v_k) \in V_{\calA \cup \{n_1\}}$, 
$$
I(v) = \{ x \in \R : x = [v_1, v_2,\ldots, v_{k-1}, v_k+y] \text{, for some } 0<y<1 \}.
$$
Given two open covers $\calC_1$ and $\calC_2$, we say that $\calC_1 \succ \calC_2$ if there exists a bijection between the disks of two covers such that the radius of each disk in $\calC_1$ is strictly larger than that the of the mapped disk in $\calC_2$. 
Moreover, denote $F_{\calA \cup \{ n_1 \}}$ as the family of canonical open covers of $\frC_{\calA \cup \{ n_1\}}$. It is easy to show that there exists a bijection $P : F_{\calA \cup \{n_1 \}} \rightarrow F_{\calA \cup \{ n_2 \}}$ between the two families such that  for each canoncial open cover $C$ of $\frC_{\calA \cup \{ n_1\}}$, we have $C \succ P(C)$. This implies the inequality for the Hausdorff dimensions $\delta_{\calA \cup \{ n_1\}} > \delta_{\calA \cup \{ n_2\}}$. 
\end{proof}

Combining Theorem \ref{thm:10000} and Lemma \ref{lemma:1003}, we then conclude

\begin{proposition}\label{prop:1004}
Among all alphabets whose corresponding determinant-one semigroup $\Gamma_{\calA}$ does not have everywhere strong approximation, the set of all odd numbers,  $\calA_{odd} =2\N-1$, has the largest Hausdorff dimension.\end{proposition}

\begin{proof}
It is easy to check that $\Gamma_{\calA_{odd}} \nequiv \text{SL}_2(\Z / 2 \Z) \pmod{2}$.
\end{proof}
Although $\calA_{odd}$ does not have everywhere strong aprroximation,  it has no local obstructions. The following Propositions give upper bounds for the Hausdorff dimensions of alphabets with local obstructions. 

\begin{proposition}\label{prop:1005}
Among all alphabets which have local obstructions and contain $1$, the alphabet $\calA_{oct} = 8\N-7 $ has the largest Hausdorff dimension.
\end{proposition}

Before we give a proof of Proposition \ref{prop:1005}, we need the next two Lemmas.

\begin{lemma}\label{lemma:1005}
Suppose we are given an integer $k \geq 2$, a residue $0 < \bar{r} \leq k $, and an alphabet $\calA= \bar{r} + k \Z$. If for all $n \geq 1$, we have $ \dd_{\calA}  \pmod{k^n} = \Z/k^n\Z$, then $\calA$ has no local obstructions. 
\end{lemma}

\begin{proof}
Given the assumption that $ \dd_{\calA}  \pmod{k^n} = \Z/k^n\Z$ for all $n \geq 1$, we hope to prove that for all integer $q$, we have $ \dd_{\calA}  \pmod{q} = \Z/q\Z$. In fact, it suffices to show that for all $q = k^n \cdot q_1$, with $(k,q_1) = 1$, and $n \geq 1$, we have $ \dd_{\calA}  \pmod{q} = \Z/q\Z.$

We now fix an arbitrary integer $q = k^n \cdot q_1$, with $(k,q_1) = 1$, and $n \geq 1$, and a residue $r \pmod{q}$. By assumption, there exists some element 
$$
\gamma = \lmx 0 & 1 \\ 1 & d_1 \\  \rmx \lmx 0 & 1 \\ 1 & d_2 \\  \rmx \cdots \lmx 0 & 1 \\ 1 & d_l \\  \rmx \in \calG_{\calA},
$$
where $\langle \gamma e_2, e_2 \rangle \equiv r \pmod{k^n}$. We now consider the following two possible cases.

\begin{enumerate}
\item $l$ is odd. For integer $1 \leq i \leq l-1$, we replace $d_{i}$ by another integer $d'_{i} = d_{i} + k^nm_{i}$ so that $d'_{i} \equiv 0 \pmod{q_1}$. Note that such integer exists since $(k,q_1) =1$. For the last index $l$, we replace $d_{l}$ by some integer $d'_{l} = d_{l} + k^n m_{l}$ so that $d'_{l} \equiv r \pmod{q_1}$. It is easy to check that the new element 
$$
\gamma' = \lmx 0 & 1 \\ 1 & d'_1 \\  \rmx \lmx 0 & 1 \\ 1 & d'_2 \\  \rmx \cdots \lmx 0 & 1 \\ 1 & d'_l \\  \rmx 
$$
still belongs to $\calG_{\calA}$. In addition, by chinese remainder theorem, we have $\langle \gamma' e_2, e_2 \rangle \equiv r \pmod{k^nq_1}$.
\item $l$ is even. Similarly, for integer $1 \leq i \leq l-2$, we replace $d_{i}$ by another integer $d'_{i} = d_{i} + k^nm_{i}$ so that $d'_{i}  \equiv 0 \pmod{q_1}$. For the last two indices, we replace $d_{l-1}$ and $d_{l}$ by some integers $d'_{l-1} = d_{l-1} + k^n m_{l-1}$ and $d'_{l} = d_{l} + k^n m_{l}$ so that $1+ d'_{l-1}d'_{l} \equiv r \pmod{q_1}$. Again, we obtain a new element $\gamma' \in \calG_{\calA}$ with $\langle \gamma' e_2, e_2 \rangle \equiv r \pmod{k^nq_1}$.
\end{enumerate}
Combining the above two cases now proves the claim. 
\end{proof}

\begin{lemma}\label{lemma:1006}
Suppose we are given an integer $k \geq 2$, a residue $0 < \bar{r} \leq k $, and an alphabet $\calA= \bar{r} + k \Z$. If for any residue $r \pmod{k}$, there exists some matrix  $\lmx a & b \\ c & d \\  \rmx \in \calG_{\calA}$ with $(c,k) = 1 $ and $d \equiv r \pmod{k}$, then for any $n \geq 1$, we have $\dd_{\calA}  \pmod{k^n}  = \Z/k^n\Z $. 
\end{lemma}

\begin{proof}
Fix an arbitrary integer $n \geq 1$, and a residue $r' \pmod{k^n}$. By assumption, there exists some element $\gamma = \lmx a & b \\ c & d \\  \rmx \in \calG_{\calA}$ such that $d \equiv r' \pmod{k}$ and $(c,k) = 1$. Moreover, similar to the proof of Lemma \ref{lemma:1001}, the matrix $\lmx 1 & k \\ 0 & 1 \\  \rmx$ is in $\calG_{\calA} \pmod{k^n}$. Consequently, there exists some integer $l$ such that 
$$
 \lmx a & b \\ c & d \\  \rmx \cdot \lmx 1 & kl \\ 0 & 1 \\  \rmx \in \calG_{\calA} \pmod{k^n},
$$
where $ckl + d \equiv r' \pmod{k^n}$. Note that such integer $l$ exists since $c$ and $k$ are coprime. Hence, we prove Lemma \ref{lemma:1006}.
\end{proof}

Now, we give the

\begin{myproof}[Proof of  Proposition \ref{prop:1005}]

Let $\calA$ be an alphabet which contains 1 and has local obstructions.  In addition, let $k^{\ast}$ be the greatest common divisor of all possible difference $|m_1 - m_2|$ with $m_1,m_2 \in \calA$. This implies that $k^{\ast} \geq 2$, and for  all integer $q$, we have $\calG_{\{ 1 + k^{\ast} \Z \}} \equiv \calG_{\calA} \pmod{q}$. Since we are looking for an upper bound of the Hausdorff dimension of $\calA$, we can assume that $\calA$ is the set of all positive integers of the form $1 + k^{\ast} \Z$, for some $k^{\ast} \geq 2$. 

When $k^{\ast} = 8$, one can see that $8\Z +4 \nsubset \frA_{\calA_{oct}} $. For the case $2 \leq k^{\ast} \leq 7$, we use Lemma \ref{lemma:1005} and Lemma \ref{lemma:1006} to show that $\calA$ has no local obstructions. In particular, when $2 \leq k^{\ast} \leq 7$, it suffices to consider powers of the matrix  $\lmx 0 & 1 \\ 1 & 1 \\  \rmx $ so that the assumption in Lemma \ref{lemma:1006} holds. Consequently, among all alphabets with local obstructions and containing $1$, the alphabet $\calA_{oct} = 8\N-7 $ has the largest Hausdorff dimension.
\end{myproof}

Similarly, using Theorem \ref{thm:10000} and Lemma \ref{lemma:1003}, we get

\begin{proposition}\label{prop:1006}
Among all alphabets which have local obstructions and do not contain 1, the alphabet $\calA_{even} = 2 \N $ has the largest Hausdorff dimension.
\end{proposition}

Finally, combining Proposition \ref{prop:1005} and Proposition  \ref{prop:1006}, we obtain the following upper bound for an alphabet $\calA$ to have local obstructions. 

\begin{theorem}\label{thm:10001}
Let $\delta_{oct}$ be the Hausdorff dimension of the alphabet $\calA_{oct} = 8\N - 7$. Let $\delta_{even}$ be the Hausdorff dimension of the alphabet $\calA_{even} = 2\N$. Denote by $\bar{\delta}$ the maximum of these two Hausdorff dimensions, $\max{\left(\delta_{oct}, \delta_{even}\right)}$. Then a finite alphabet $\calA$ has no local obstructions if $\delta_{\calA} > \bar{\delta}$. 
\end{theorem}

\begin{remark}
One can also see that there exists a finite alphabet $\calA$ with local obstructions and Hausdorff dimension arbitrarily close to $\bar{\delta}$. For example, we can take finite truncations of the alphabets $\calA_{oct}$ or $\calA_{even}$.
\end{remark}

For $\delta_{oct}$ and $\delta_{even}$, we give estimates using the algorithm provided by \cite{ETS:86119}, see Table \ref{table1}. While using the Hausdorff dimensions of finite truncations of $\calA_{oct}$ and $\calA_{even}$ to approximate $\delta_{oct}$ and $\delta_{even}$ is not a rigorous approach, the following data suggests that $5/6 > \bar{\delta}$. That is, we believe that when $\delta_{\calA} > 5/6$, the alphabet $\calA$ should have no local obstructions.  

\begin{table}
\caption {Estimates of The Hausdorff Dimensions for Subsets of $\calA_{oct}$ and $\calA_{even}$}\label{table1}
\begin{tabular}{l*{6}{c}r}
Alphabet     &&&&&&&   Hausdorff dimension \\
\hline
$\{1,9,17,25,33,41\}$ &&&&&&& $\approx 0.472$ \\ 
$\{1,9,17,25, \ldots, 129  \}$ &&&&&&& $\approx 0.55$ \\
$\{1,9,17,25,\ldots, 201  \}$ &&&&&&& $\approx 0.56$ \\
$\{2,4,6,8,\ldots, 20  \}$ &&&&&&& $\approx 0.59$ \\
$\{2,4,6,8,\ldots, 200  \}$ &&&&&&& $\approx 0.68$ \\
$\{2,4,6,8,\ldots, 300  \}$ &&&&&&& $\approx 0.69$ \\
$\{2,4,6,8,\ldots, 802  \}$ &&&&&&& $\approx 0.70$ \\

\end{tabular}
\end{table}

\bigskip 
\bigskip 
\bigskip

\bibliographystyle{alpha}
\bibliography{Zaremba3}

\end{document}